\documentclass[11pt,english]{article}
\usepackage[margin= 0.65in ,bottom=30mm, top= 16mm]{geometry}

\usepackage{amsthm}
\usepackage{amsmath}
\usepackage{amssymb}
\usepackage{setspace}
\usepackage{mathtools}
\usepackage{verbatim}
\usepackage{csquotes}
\usepackage{graphicx}
\usepackage[hidelinks]{hyperref}
\usepackage{thm-restate}
\usepackage{cleveref}

\usepackage{enumitem}
\setlist[itemize]{leftmargin=*} 
\setlist[enumerate]{leftmargin=*}
\usepackage{framed}
\usepackage{subcaption}
\usepackage{floatrow}
\usepackage[T1]{fontenc}
\usepackage{mathdots}
\usepackage{xcolor}
\usepackage{diagbox}
\usepackage{colortbl}

\theoremstyle{plain}

\newtheorem{theorem}{Theorem}[section]
\newtheorem{claim}[theorem]{Claim}

\newtheorem{lemma}[theorem]{Lemma}

\newtheorem{conjecture}[theorem]{Conjecture}
\newtheorem{observation}[theorem]{Observation}

\theoremstyle{definition}

\newtheorem{defn}[theorem]{Definition}
\newtheorem*{defn*}{Definition}

\expandafter\def\expandafter\normalsize\expandafter{%
    \normalsize
    \setlength\abovedisplayskip{4pt}
    \setlength\belowdisplayskip{4pt}
    \setlength\abovedisplayshortskip{4pt}
    \setlength\belowdisplayshortskip{4pt}
}

\usepackage[square,sort,comma,numbers]{natbib}
\newcommand{\calG}{\mathcal{G}}

\newcommand{\calI}{\mathcal{I}}

\newcommand{\calP}{\mathcal{P}}

\def\eps {\varepsilon}

\newcommand{\Bin}{\mathrm{Bin}}

\newcommand{\Gr}{G_\mathrm{red}}
\newcommand{\indsizeramsey}{\hat{R}_{ind}^k}

\newcommand{\cI}{\mathcal{I}}

\newcommand\blfootnote[1]{%
  \begingroup
  \renewcommand\thefootnote{}\footnote{#1}%
  \addtocounter{footnote}{-1}%
  \endgroup
}

\renewcommand{\Pr}{\mathbb{P}}

\title{\vspace{-0.8cm}Effective bounds for induced size-Ramsey numbers of cycles}

\author{
Domagoj Brada\v{c}
\and
Nemanja Dragani\'{c}
\and 
Benny Sudakov
}

\date{}

\begin{document}

\maketitle
\begin{abstract}
The induced size-Ramsey number \blfootnote{Department of Mathematics, ETH, Z\"urich, Switzerland. Research supported in part by SNSF grant 200021\_196965.\\ Email: \textbf{\{domagoj.bradac, benjamin.sudakov, nemanja.draganic\}@math.ethz.ch}.}$\hat{r}_\text{ind}^k(H)$ of a graph $H$ is the smallest number of edges a (host) graph $G$ can have such that for any $k$-coloring of its edges, there exists a monochromatic copy of $H$ which is an induced subgraph of $G$. 
In 1995, in their seminal paper, Haxell, Kohayakawa and {\L}uczak showed that for cycles, these numbers are linear for any constant number of colours, i.e.,  $\hat{r}_\text{ind}^k(C_n)\leq Cn$ for some $C=C(k)$. The constant $C$ comes from the use of the regularity lemma, and has a tower type dependence on $k$. In this paper we significantly improve these bounds, showing that $\hat{r}_\text{ind}^k(C_n)\leq O(k^{102})n$ when $n$ is even, thus obtaining only a polynomial dependence of $C$ on $k$. We also prove $\hat{r}_\text{ind}^k(C_n)\leq e^{O(k\log k)}n$ for odd $n$, which almost matches the lower bound of $e^{\Omega(k)}n$. Finally, we show that the ordinary (non-induced) size-Ramsey number satisfies $\hat{r}^k(C_n)=e^{O(k)}n$ for odd $n$. This substantially improves the best previous result of $e^{O(k^2)}n$, and is best possible, up to the implied constant in the exponent. To achieve our results, we present a new host graph construction which, roughly speaking, reduces our task to finding a cycle of approximate given length in a graph with local sparsity.
\end{abstract}

\section{Introduction}

The Ramsey number $r^k(H)$ of a graph is the smallest integer $n$ such that every $k$-coloring of the edges of $K_n$ contains a monochromatic copy of $H$. The notion of Ramsey numbers is one of the most central notions in combinatorics and it has been studied extensively since Ramsey \cite{Ramsey1930} showed their existence for every graph $H$.
Motivated by this definition, we say that a graph $G$ is $k$-Ramsey for a graph $H$ if any $k$-coloring of the edges of (the \emph{host graph}) $G$, contains a monochromatic copy of $H$, and we write $G\xrightarrow{k}H$. Using this notation, we have that $r^k(H)=\min\{|V(G)|: G\xrightarrow{k}H\}.$

The notion of Ramsey numbers is measuring the minimality of the host graph in terms of the number of vertices. Are there graphs $G$ with significantly less edges than the clique on $r^k(H)$ vertices that are $k$-Ramsey for $H$? This general question is captured by the notion of size-Ramsey numbers introduced in 1978 by Erd\H{o}s, Faudree, Rousseau and Schelp~\cite{erdHos1978size}. The size-Ramsey number of a graph $H$ is defined as $\hat{r}^k=\min\{E(G)| G\xrightarrow{k} H\}$. In the last few decades, there has been extensive research on this notion, see, e.g., \cite{CFS}.

One of the main goals is to understand which classes of graphs have size-Ramsey numbers which are linear in their number of edges. Beck \cite{beck1983size} showed that this is true for paths, which was later extended to all bounded-degree trees by Friedman and Pippenger \cite{friedman1987expanding}. It is also known that logarithmic subdivisions of bounded degree graphs have linear size-Ramsey numbers \cite{draganic2022rolling}, as well as bounded degree graphs with bounded treewidth \cite{kamcev2021size, Berger21}. Given all of the mentioned results, it might be tempting to assume that all graphs of bounded degree have linear size-Ramsey numbers. In an elegant paper of R\"{o}dl and Szemer\'{e}di \cite{rodl2000size}, it was shown that this is not the case. Indeed, they showed that there exist $n$-vertex cubic graphs which have size-Ramsey numbers at least $n\log^c n$, for a small constant $c>0$. This bound has only very recently been improved to $c n e^{c \sqrt{\log n}}$ for some $c > 0$ by Tikhomirov \cite{tikhomirov2022bounded}. For more results see \cite{draganic2022size} and references therein.

A related studied notion is that of induced size-Ramsey numbers. Given a graph $H$, the induced size-Ramsey number $\hat{r}^k_{ind}(H)$ is the smallest number of edges a graph $G$ can have such that any $k$-coloring of $G$ contains a monochromatic copy of $H$ which is an induced subgraph of $G$. The existence of these numbers is an important generalisation of Ramsey's theorem, proved independently by Deuber \cite{Deuber},  Erd\H{o}s, Hajnal, and Pósa \cite{ErHaPo}, and Rödl \cite{Rodl73}.
Naturally, this concept is much harder to understand for most classes of target graphs $H$ and much less precise bounds are known than for the (non-induced) size-Ramsey number.

Indeed, already for bounded degree trees we know that the size-Ramsey number is linear in their number of vertices, whereas for its induced counterpart we have no good bounds while we have every reason to believe that the answer should also be linear. Further, the best general upper bound on $\hat{r}^2_{ind}(H)$ for $n$-vertex graphs $H$ is obtained by Conlon, Fox and Sudakov \cite{kohayakawa1998induced}, and is of the order $2^{O(n\log n)}$, while Erd\H{o}s \cite{erdos1975problems} conjectured that $\hat{r}^2_{ind}(H)\leq 2^{cn}$. In comparison, the bound for Ramsey numbers (and hence also for size-Ramsey numbers) is known to be exponential in the number of vertices of the target graph. Further, it is known that the size-Ramsey number of $n$-vertex graphs with degree bounded by a constant $d$, is between $n e^{\Omega(\sqrt{\log n})}$ and $O(n^{2-\frac{1}{d}+\varepsilon})$, proven by Tikhomirov \cite{tikhomirov2022bounded}, and by  Kohayakawa, Rödl, Schacht, and Szemerédi \cite{kohayakawa2011sparse}, respectively. On the other hand, the best upper bound on the induced size-Ramsey number of these graphs, proved by Fox and Sudakov \cite{fox2008induced} is of the order $n^{O(d\log d)}$, while the best lower bound is still the bound for the (non-induced) Ramsey number of those graphs, which is often the state of the art for such questions.

For paths it is known that $\Omega(k^2)n\leq \hat{r}^k(P_n)\leq O(k^2\log k)n$ (see \cite{dudek2017some, krivelevich2019expanders} for the lower bound and \cite{krivelevich2019long,dudek2015alternative} for the upper bound). In the induced case, by a recent result of Dragani\'c, Krivelevich and Glock \cite{draganic2022short}, we have that $\hat{r}^k_{ind}(P_n)\leq O(k^3\log^4k)n$. For cycles, the discrepancy between the size-Ramsey and the induced size-Ramsey number is significantly larger.
Indeed, by a recent result of Javadi and Miralaei \cite{javadi2023multicolor}, which improved another recent result by Javadi, Khoeini, Omidi and Pokrovskiy \cite{javadi2019size}, we have $r^k(C_n)=O(k^{120}\log^2k)n$ for even $n$, and $r^k(C_n)=O(2^{16k^2+2\log k})n$ for odd $n$. On the other hand, the only 
known upper bound on the induced size-Ramsey numbers of cycles was obtained in the seminal paper of Haxell, Kohayakawa and {\L}uczak \cite{haxell1995induced}. Their proof uses a technically very involved argument relying on the use of the Sparse Regularity lemma and therefore 
shows that $r_{ind}^k(C_n) \le Cn$ where $C=C(k)$ has a tower type dependence on $k$.

In this paper, we prove the following theorem which quite significantly improves the tower-type bounds of Haxell, Kohayakawa and {\L}uczak.
\begin{theorem}\label{thm:main}
        For any integer $k \ge 1,$ there exists $n_0(k)$ such that for all $n \ge n_0(k),$ the following holds.
        \begin{enumerate}[label=\alph*)]
        \item \label{induced-even} If $n$ is even, then $\hat{r}_{ind}^k(C_n)=O(k^{102})n.$
        \item \label{induced-odd} If $n$ is odd, then $\hat{r}_{ind}^k(C_n)=e^{O(k\log k)}n$.
    \end{enumerate}
\end{theorem}

While the focus of this paper is on induced size-Ramsey numbers of cycles, our method can be also used to substantially improve the upper bound for the non-induced case as well. Our next result gives an essentially tight estimate for the size-Ramsey numbers of odd cycles.
\begin{theorem}\label{thm:non-induced}
        For any integer $k \ge 1,$ there exists $n_0(k)$ such that for all $n \ge n_0(k),$ we have $\hat{r}^k(C_n)=e^{O(k)}n$.
\end{theorem}

The best known lower bound for size-Ramsey numbers of even cycles comes from the bound for paths, which is of the order $\Omega(k^2)n$ \cite{dudek2017some, krivelevich2019long}. In the odd case, there is a simple construction of a coloring which gives a lower bound of $2^{k-1}n$ (see \cite{javadi2023multicolor}), showing that the second result in Theorem~\ref{thm:main} is tight up to an $O(\log k)$ factor in the exponent, while the bound in Theorem~\ref{thm:non-induced} is tight up to a constant factor in the exponent.

We remark that, as in \cite{haxell1995induced}, our proofs can easily be adapted to provide monochromatic induced cycles of all (even) lengths between $C \log n$ and $n$ for some constant $C$ depending only on $k$.
We also note that our bound on the size-Ramsey number of even cycles $\hat r^k(C_n)\leq \hat r^k_{ind} (C_n)=O(k^{102})n$ can be further improved significantly, using the same methods, but we chose not to present that here.

We systematically ignore floor and ceiling signs whenever they are not crucial for the argument. All logarithms are base $e$ unless otherwise specified. We make no serious attempt to optimize the constants in our proofs. We use standard graph theoretic notation throughout. We denote by $\delta(G)$ and $\Delta(G)$ the minimum and maximum degree of a graph $G$ and we use $V(G)$ and $E(G)$ to denote its vertex and edge set respectively. We also denote $v(G) = |V(G)|$ and $e(G)=|E(G)|.$

\section{Proof outline}\label{sec:outline}
The main idea behind our proof is the following: consider a binomial random graph $G(N, C/N),$ where $N = C'n$ and $C, C'$ are appropriately chosen large constants. Let $G$ be adversarially $k$-edge-colored. Then, it is easier to find an induced monochromatic cycle of length in $[0.9n, 1.1n],$ say, then of length precisely $n$. Our host graph is constructed to take advantage of this.

In the rest of the outline we focus on the proof of the induced odd case (Theorem~\ref{thm:main}~\ref{induced-odd}) and at the end we outline the changes needed for the other two statements.

Given $k$, we find a fixed ``gadget'' graph $F = F(k)$ which is $k$-induced-Ramsey for a $5$-cycle. We denote $s = v(F).$
We construct an $s$-uniform $N$-vertex hypergraph $H$ by taking $CN$ random hyperedges. We clean $H$ so it does not have any short Berge cycles (see Definition~\ref{def:berge}) so, in particular, it is linear. Then we construct our host graph $\Gamma$ by placing an isomorphic copy of $F$ inside every hyperedge of $H.$ By definition, inside every copy of $F$, there is a monochromatic induced copy of $C_5.$ The main object we work with will be an auxiliary $k$-edge-coloured graph $G$ on the same vertex set as $\Gamma$. For each placed copy of $F$ in $\Gamma$, in $G$ we put an edge between a single pair of vertices which are at distance $2$ in one of the induced monochromatic copies of $C_5$ in the  copy of $F$, and colour this edge with the colour of that cycle. 

Now, suppose we find a monochromatic, say red, cycle $Q$ of length $\ell \in [n/3, n/2]$ in $G.$ By definition, each edge of $Q$ corresponds to an induced $5$-cycle in $\Gamma,$ where the endpoints of the edge are at distance $2$ in the cycle. For each of these $5$-cycles, we can choose either a path of length $2$ or a path of length $3$ in $G$ to obtain a red cycle $Q'$ of length exactly $n$ in $\Gamma$ (see Figure~\ref{fig:transforming cycle}). The main technical difficulty is in obtaining certain properties of $Q$ such that the resulting cycle $Q'$ is induced in $\Gamma.$ 

More precisely, the following will be sufficient. Recall that every edge $e \in E(G)$ comes from a hyperedge in $H$ which we denote by $h(e).$ Suppose $Q$ is a cycle in $G$ with edges $e_1, \dots, e_\ell$ such that no hyperedge apart from $h(e_1), \dots, h(e_\ell)$ in $H$ intersects $\bigcup_{i \in [\ell]} h(e_i)$ in more than one vertex. Further, suppose that each $h(e_i)$ only intersects $h(e_{i-1})$ and $h(e_{i+1})$ among the mentioned hyperedges. Then, it is not difficult to see that the cycle $Q'$ obtained as above is induced in $\Gamma.$ We will call such a cycle $Q$ \emph{good}.

Let us now explain how to find an induced monochromatic cycle of length between $n/2$ and $n/3$ in a $k$-edge-colored graph $G \sim \calG(N, C/N)$ with $N = C'n$ for some large constants $C, C'.$ Our real task is more involved as we require a stronger condition on the found cycle as discussed above, since we are not working with a binomial random graph. However, most of the ideas can be described through the lens of this simpler problem.

We now sketch how to find a monochromatic induced cycle of length between $n/2$ and $n/3$ in $G \sim \calG(N, C/N).$ The proof strategy is illustrated in Figure~\ref{fig:proof-sketch}. By standard results, it is not difficult to clean $G$ without losing many edges, so that it has no cycles of length $O(1)$. Further, we also know that it is locally sparse, that is, all sets $U$ of size $|U| \le \eps N$ span at most $\frac{3}{2} |U|$ edges, where $\eps > 0$ is some constant depending on $C.$ We consider the subgraph corresponding to the densest colour class, say red and using a result of Krivelevich~\cite{krivelevich2019expanders}, we find inside it a large expanding subgraph $G'$. Dragani\'{c}, Glock and Krivelevich~\cite{draganic2022short} showed using a modified DFS algorithm that under the given assumptions, $G'$ has a red induced path of length $2n/5$ and we adapt their argument to our setting. Given such a red induced path of length $2n/5,$ from the endpoints we construct two trees $T_1, T_2$ each of depth $O(\log N)$ and with $\Omega(\eps N)$ leaves. Moreover, we do it in such a way that any path containing the initial endpoints is good, i.e. if there is a red edge connecting two vertices in different trees, it closes a good cycle in $G'$. Let $W = V(P) \cup V(T_1) \cup V(T_2)$ and remove from it a large constant number of the last layers in $T_1$ and $T_2$, so that the resulting $W$ is small enough compared to the leaf sets of $T_1$ and $T_2$. Denote by $R_1$ and $R_2$ the vertices in the deleted layers in $T_1$ and $T_2$, respectively. Finally, using the expanding properties of $G',$ we may expand from the sets $R_1$ and $R_2$, while avoiding vertices which are incident to $W$ until the two balls around $R_1$ and $R_2$ of large enough constant diameter intersect, and thus we close a cycle of desired length. Using the girth assumption on our graph it is not difficult to show that this cycle is induced.

Let us now comment on the differences in the proofs for the three different statements. In the odd induced case, we can take $F$ to be Alon's celebrated construction of a dense pseudorandom triangle-free graph on $e^{\Theta(k \log k)}$ vertices. We will prove that, every $k$-edge-colouring of that graph will contain an induced monochromatic $C_5.$ However, when $n$ is even, we can instead take $F$ to be $k$-induced-Ramsey for a $6$-cycle with only $O(k^6)$ vertices by taking a sufficiently dense bipartite $C_4$-free graph. Again, in each copy of $F,$ we find a monochromatic induced $6$-cycle and connect two vertices at distance $2$ on the cycle to form our auxiliary graph. The same argument as above shows that given a monochromatic cycle of length $\ell$ in the auxiliary graph $G,$ we can find a monochromatic cycle of any even length between $2\ell$ and $4\ell$ in $\Gamma.$ Finally, for the odd non-induced case, we can take $F$ to be the complete graph on $2^k + 1$ vertices. It is easy to see that any $k$-edge coloring of that graph has a monochromatic odd cycle. For simplicity, we take the most common length $L$ among those cycles, and for each of these $L$-cycles, we connect two vertices at distance $(L-1)/2$ on the cycle to form the auxiliary graph. Then, a monochromatic good cycle of length between $2n / (L-1)$ and $2n / (L+1)$ in the auxiliary graph yields a monochromatic cycle of length $n$ in our host graph. This required extra precision in the length of the good cycle in the auxiliary graph will only cost us a factor of $2^{O(k)}$ in the number of copies of $F$ we use in our construction.

The rest of the paper is structured as follows. In Section 3, we provide constructions of the small gadget graphs. In Section 4, we prove the properties of the random hypergraph $H$ mentioned above and present the construction of our host graph. In Section 5, we show how to find the desired induced cycle in the host graph with the proof split into four subsections corresponding to different stages of the proof. Finally, we end with some concluding remarks in Section 6.

\begin{figure}[h]
    \caption{Transforming an $8$-cycle in the auxiliary graph (thick red edges) into a $21$-cycle in the original graph by using $5$ paths of length $3$ and $3$ paths of length $2$.}\label{fig:transforming cycle}
    \includegraphics[scale=0.4]{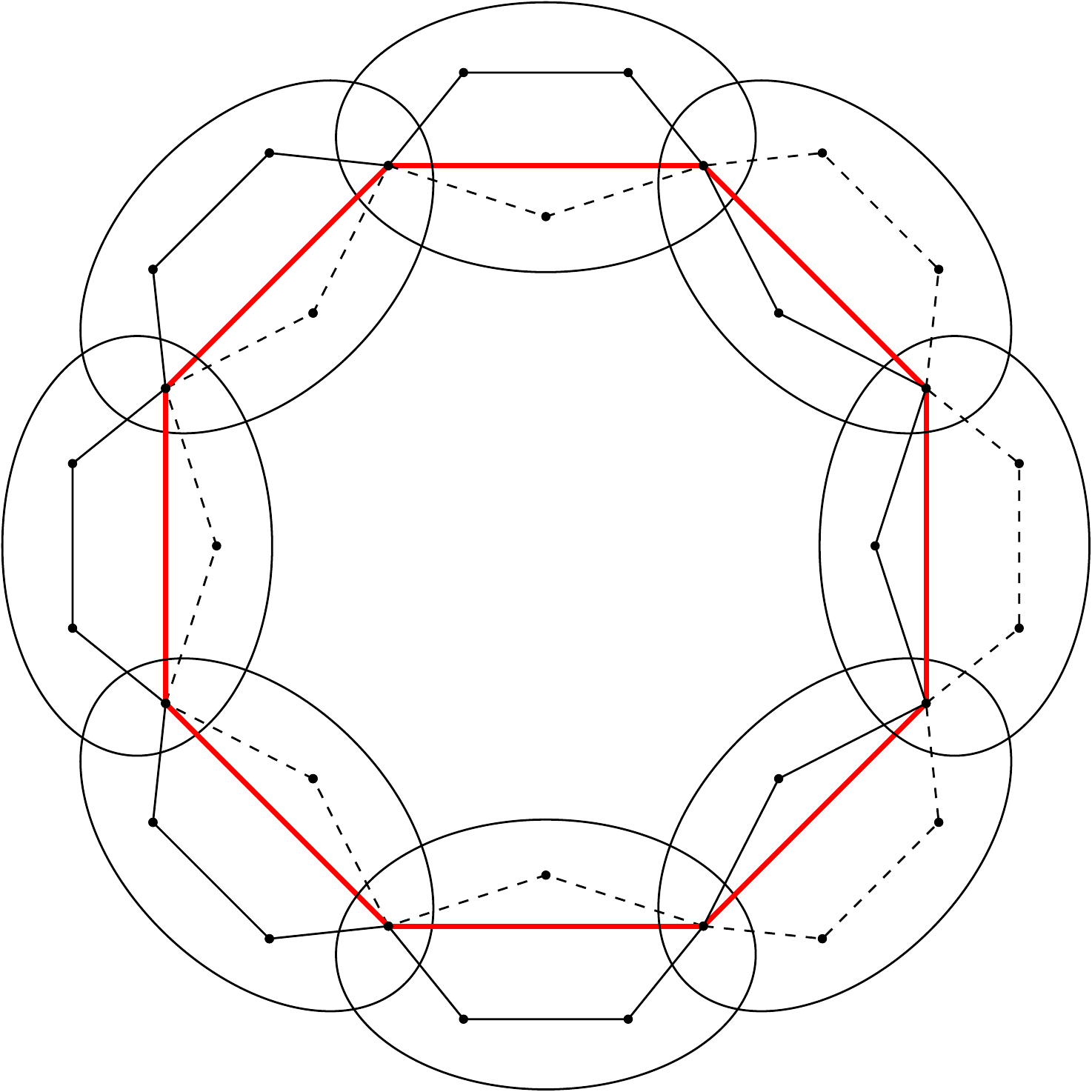}
\end{figure}

\begin{figure}[h]
    \caption{Building an induced cycle} \label{fig:proof-sketch}
    \includegraphics[scale=0.7]{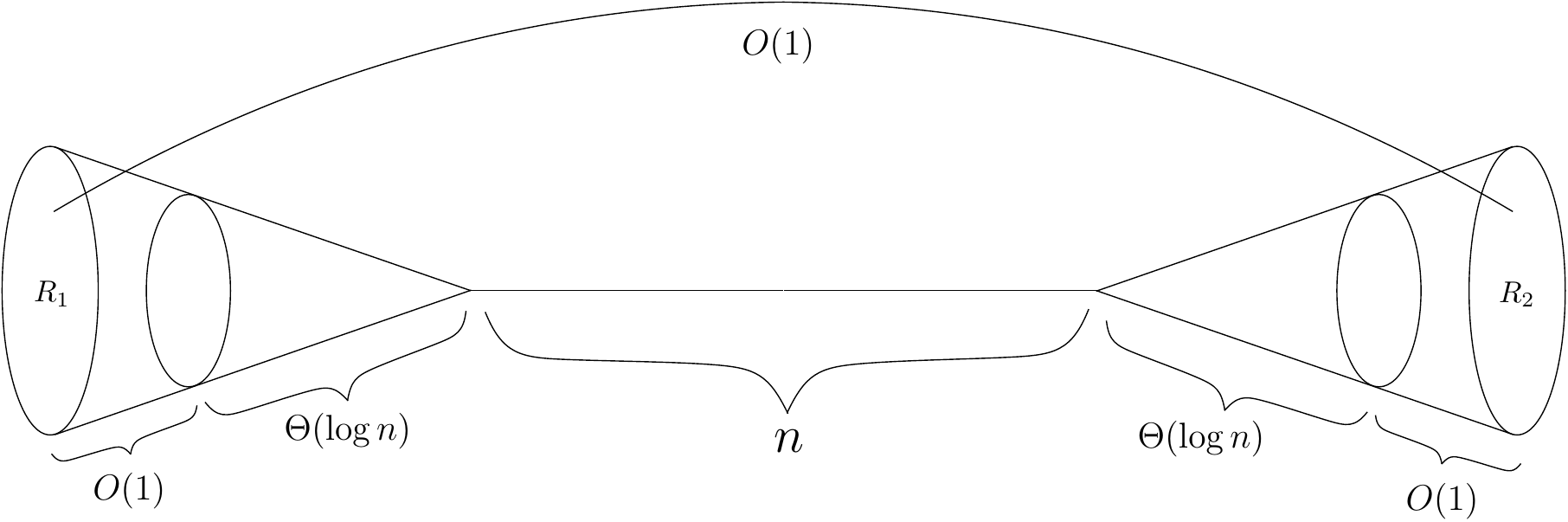}
\end{figure}

\section{The multicolour induced size-Ramsey numbers of short cycles}
In this section we give upper bounds for the multicolour induced size-Ramsey number of short cycles. The graphs providing these upper bounds will be used as the building blocks of our host graph. For the proof of Theorem~\ref{thm:non-induced}, we simply take the complete graph on $2^k + 1$ vertices, the reason for this choice being the following simple lemma.

\begin{lemma} \label{lem:any-odd-cycle}
    Any $k$-edge-colouring of the complete graph on $2^k + 1$ vertices contains a monochromatic odd cycle.
\end{lemma}
\begin{proof}
    For the sake of contradiction, suppose there is a $k$-edge-colouring of $K_{2^k+1}$ with no monochromatic odd cycle. Hence, for each colour $i \in [k],$ the graph induced by colour $c$ must be bipartite with a bipartition $A_i \cup B_i.$ Since there are more than $2^k$ vertices, by the pigeonhole principle there are two vertices $x, y$ such that $x \in A_i \iff y \in A_i$ for each $i \in [k]$. Let $j$ denote the colour of the edge $xy.$ Then, $xy$ is an edge of colour $j$ connecting two vertices which are in the same part of the bipartition $A_j \cup B_j,$ a contradiction.
\end{proof}

\subsection{The $6$-cycle}
The existence of a small host graph for the induced $6$-cycle is a simple consequence of the following two well-known results in extremal graph theory.

\begin{theorem}[Bondy, Simonovits \cite{bondy1974cycles}] \label{thm:bondy-simonovits}
    $\mathrm{ex}(n, C_6) = O(n^{4/3}).$
\end{theorem}

\begin{theorem}[Erd\H{o}s, R\'{e}nyi \cite{erdos-renyi}] \label{thm-c4-free-graph}
    For all large enough $n,$ there exists a bipartite $C_4$-free graph with $n$ vertices and $(1/4 + o(1))n^{3/2}$ edges.
\end{theorem}

\begin{lemma}\label{lem:gadget even}
     There exists a positive constant $C$ such that for any positive integer $k,$ there exists a graph $F_0(k)$ with $(Ck)^6$ vertices and $(Ck)^9 / 6$ edges such that any $k$-edge-colouring contains a monochromatic $6$-cycle forming an induced subgraph of $F_0(k).$ In particular, $\indsizeramsey(C_6) = O(k^9).$
\end{lemma}
\begin{proof}
    Clearly we may assume that $k$ is large enough. Let $n = (Ck)^6$ where $C$ is a large constant to be chosen later. Let $F_0(k)$ be a bipartite $C_4$-free graph on $n$ vertices with at least $n^{3/2}/8$ edges given by Theorem~\ref{thm-c4-free-graph}. Note that $F$ has girth at least $6$ because it is bipartite and $C_4$-free. Consider an arbitrary $k$-edge colouring of $F_0(k).$ Then, there exists a colour, say red, which contains at least $e(F_0(k)) / k \ge (C/8) n^{4/3}$ edges. Taking $C/8$ to be larger than the implied constant in Theorem~\ref{thm:bondy-simonovits}, we obtain that there exists a red $6$-cycle in $F_0(k).$ This cycle is necessarily induced since $F_0(k)$ has girth at least $6,$ thus completing the proof.
\end{proof}

\subsection{The $5$-cycle}
Following Alon, a graph $F$ is said to be an $(n, d, \lambda)$-graph if $F$ has $n$ vertices, is $d$-regular and all but the largest eigenvalue of $F$ are at most $\lambda$ in absolute value. As our host graph exhibiting the upper bound for $\indsizeramsey(C_5),$ we will use Alon's celebrated construction of a dense pseudorandom triangle-free graph.

\begin{theorem}[Alon \cite{alon1994explicit}] \label{thm:alon-graph}
    For every positive integer $t$ not divisible by $3$, there exists a triangle-free $(n, d, \lambda)$-graph with $n = 2^{3t}, d = (1/4 + o(1))n^{2/3}, \lambda = (9 + o(1))n^{1/3}.$
\end{theorem}

In the proof we use the well known result about $(n,d,\lambda)$-graphs known as the Expander Mixing Lemma. Given a graph $G$ and two sets $S, T \subseteq V(G),$ we denote by $e_G(S, T)$ the number of pairs $(u, v) \in S \times T$ such that $uv \in E(G).$ Note that, by definition, the edges inside $S \cap T$ are counted twice.
\begin{lemma}[e.g. \cite{krivelevich2006pseudo}] \label{lem:expander-mixing}
    Let $G$ be an $(n, d, \lambda)$-graph. Then, for any two subsets $S, T \subseteq V(G),$
    \[ |e_G(S, T) - \frac{d|S||T|}{n}| \le \lambda \sqrt{|S||T|}. \]
\end{lemma}

\begin{lemma}\label{lem:gadget odd}
    $\indsizeramsey({C_5}) = e^{O(k \log k)}.$
\end{lemma}
\begin{proof}
    We may assume that $k$ is sufficiently large. Let $F = F(k)$ be a triangle-free $(n, d, \lambda)$ graph, where $(20k)^{48k} \le n \le 64 (20k)^{48k}, d = (1/4 + o(1))n^{2/3}, \lambda = (9 + o(1))n^{1/3}$ whose existence is given by Theorem~\ref{thm:alon-graph}. We will prove the following claim.
    
    \begin{claim} \label{cl:5-cyc}
        Let $U \subseteq V(F)$ satisfy $|U| > n / (20k)^{k-1}$ and suppose that the edges of $F[U]$ are coloured with $t$ colours, where $1 \le t \le k$. Then, $U$ contains a monochromatic $5$-cycle or there exists a set $U' \subseteq U$ with $|U'| \ge |U| / (20k)$ such that at most $t-1$ colours appear on the edges of $F[U'].$
    \end{claim}
    
    The Lemma follows easily from this claim. Indeed, note that since $F$ is triangle-free, every $5$-cycle is induced. Now, consider an arbitrary $k$-edge-colouring of $F.$ Assume we have a set $U_i \subseteq V(F), 0 \le i \le k-1$ such that $|U_i| \ge n/(20k)^i$ and the edges of $F[U_i]$ are coloured in $k-i$ colours. By the claim, we either find a monochromatic $5$-cycle inside $F[U_i],$ in which case we are done, or a subset $U_{i+1} \subseteq U_i$ of size at least $n / (20k)^{i+1}$ with at most $k-i-1$ colours on the edges of $F[U_{i+1}].$ Starting with $U_0 = V(F),$ we either find a monochromatic $5$-cycle or after $k$ steps we get an independent set $U_k \subseteq V(F)$ of size at least $n/(20k)^k.$ However, by Lemma~\ref{lem:expander-mixing}, $e(U_k) \ge (d|U_k|^2 / n - \lambda |U_k|) / 2 > 0,$ since $|U_k| \ge n/(20k)^k > n \lambda / d.$ Hence, $U_k$ cannot be an independent set, so we have found a monochromatic $5$-cycle as required.
    
    \begin{proof}[Proof of Claim~\ref{cl:5-cyc}]
    Let $F'' \subseteq F[U]$ be the subgraph corresponding to the most frequent colour in $F[U]$ and let $F'$ be an induced subgraph of $F''$ with minimum degree at least 
    \[ d' \coloneqq \delta(F') \ge \frac{e(F'')}{|U|} \ge \frac{e(F[U])}{k|U|} \ge \frac{d|U|^2/n - \lambda |U|}{2k|U|} \ge \frac{|U|}{32k n^{1/3}}, \] where we used Lemma~\ref{lem:expander-mixing} and that $|U| \ge n/(20k)^{k-1} > n^{47/48}.$  Let $v \in V(F')$ be arbitary, let $A = N_{F'}(v)$ and $B = N_{F'}(A) \setminus \{v\},$ so $A$ is the first and $B$ the second neighbourhood of $v$ in $F'.$ Assume there is an edge $xy \in F'[B].$ Then, by definition, there exist vertices $x', y'$ in $A$ such that $xx', yy' \in E(F').$ Since $F'$ is triangle-free, we have $x' \neq y'$ and therefore, $vx'xyy'$ forms a $5$-cycle in $F',$ that is, a monochromatic $5$-cycle in $F[U].$ Now, assume that $e_{F'}[B] = 0.$ This means that $F[B]$ has at most $k-i-1$ colours on its edges, so we take $U' = B.$ What is left to show is that $|B| \ge |U| / (20k).$ Note that since $F'$ is triangle-free, $A$ is an independent set. Hence, $\sum_{a \in A} d_{F '}(a) = e(A, B) + |A|,$ implying that 
    \[ e(A, B) \ge |A| (d' - 1) \ge d'^2 / 2 \ge \frac{1}{2} \left(\frac{n^{47/48}}{32kn^{1/3}}\right)^2 = \frac{n^{31/24}}{2(32k)^2}. \]
    On the other hand, by Lemma~\ref{lem:expander-mixing}, $e(A, B) \le \frac{d}{n}|A||B| + \lambda \sqrt{|A||B|}.$ Note that $\lambda \sqrt{|A||B|} \le \lambda \sqrt{dn} < 5 n^{7/6} < \frac{1}{100} e(A,B)$, for large enough $k$. Combining, we have $\frac{d}{n}|A||B| \ge \frac{99}{100} e(A, B)$ and, therefore,
    \[ |B| \ge \frac{99}{100} \frac{n}{d}(d' - 1) \ge \frac{1}{2} 4n^{1/3} \frac{|U|}{32k n^{1/3}} > \frac{|U|}{20k}, \]
    as needed.
    \end{proof}
\end{proof}

\section{Host graph}\label{sec:host}
In order to construct our host graph, we will need a few basic definitions from hypergraph theory. We start with the following definition of a cycle in a hypergraph.
\begin{defn}[Berge cycle] \label{def:berge}
    Given a hypergraph $H,$ a \emph{Berge cycle} of length $t \ge 2$ is an alternating sequence of distinct vertices and edges $v_0, e_0, v_1, e_1, {\dots}, v_{t-1}, e_{t-1}$ such that $v_i \in e_{i-1}, e_i,$ where indices are modulo $t.$
\end{defn}
The \emph{girth} of a hypergraph is defined to be the length of its shortest Berge cycle.

\begin{defn}[Intersection graph]
    Given a linear hypergraph $H,$ we define its intersection graph $\calI(H)$ whose vertex set is the set of hyperedges of $H$ and where two hyperedges are adjacent if they share a vertex in $H.$ Additionally, each edge in $\calI(H)$ is labelled by the corresponding vertex in $V(H).$
\end{defn}

\begin{defn}[Sunflower cycle]
    Given a hypergraph $H$ and its intersection graph $\calI = \calI(H),$ a \emph{sunflower cycle} in $\calI$ is a cycle in which all the edges are labelled with the same vertex in $V(H).$
\end{defn}

Let us explain the motivation for the above two definitions. As described in the proof outline, our host graph construction begins by taking an $s$-uniform hypergraph $H$ on $N$ vertices with $CN$ random hyperedges, where $C = C(k)$. Its intersection graph $\calI(H)$ will be used to argue that a certain monochromatic cycle we find in the host graph is induced. For our arguments, we  require that $\calI(H)$ has local sparsity properties analogous to what is likely to hold in $\calG(N, C/N).$ More precisely, in $\calG(N, C/N),$ with high probability, there is no subgraph on at most $\alpha N$ vertices with average degree at least say $\frac{8}{3},$ where $\alpha > 0$ depends only on $C.$ However, such a statement does not hold for $\calI(H).$ Indeed, the set of hyperedges of $H$ containing a fixed vertex $v$ forms a clique in $\calI(H)$. Hence, it is not true that there are no subgraphs of $\calI(H)$ on a small set of vertices with average degree at least $8/3$ as in the graph case. However, it is true if we only consider subgraphs of $\calI(H)$ which do not contain sunflower cycles, i.e. subgraphs where from every clique in $\calI(H)$ corresponding to a set of hyperedges containing a fixed vertex we are only allowed to take a set of edges spanning a forest. The following lemma proves this (see \ref{no-dense-set}) as well as several other simple properties of $H$ (after small alterations) that we will use in our proof.

\begin{lemma} \label{lem:hypergraph}
    Let $s, g \ge 4$ be given positive integers, let $C \ge 1$ and set $\alpha = 10^{-6} C^{-3}s^{-8}.$ Then, for all large enough $N,$ there is an $s$-uniform hypergraph $H$ on $N$ vertices such that the following properties are satisfied:
    
    \begin{enumerate}[label=(P\arabic*)]
        \item \label{many-edges} $e(H) \in [CN/2, CN].$
        \item \label{max-degree} $\Delta(H) \le 8Cs.$
        \item \label{girth} $H$ has girth larger than $g$ (in particular, $H$ is linear).
        \item \label{no-dense-set} There is no subgraph $\calI' \subseteq \calI(H)$ such that $v(\calI') < \alpha N,$ $e(\calI') > \frac{4}{3} v(\calI')$ and there is no sunflower cycle in $\calI'.$ 
        \item \label{no-dense-aux-graph-set} For any $A \subset V(H)$ with $|A| \le \alpha N,$ there are at most $2 |A|$ hyperedges in $H$ which intersect $A$ in at least two vertices.
    \end{enumerate}
\end{lemma}
\begin{proof}
    Let $m = CN.$ We construct an $s$-uniform hypergraph $H_0$ on the vertex set $V = [N]$ with edges $B_1, \dots, B_m$ which are independently chosen subsets of $V$ of size $s,$ removing duplicate edges. For any Berge cycle of length at most $g$ in $H_0,$ we remove an arbitrary edge of the cycle. Additionally, we remove all edges incident to a vertex of degree more than $8Cs.$ Let $H$ denote the resulting hypergraph. Note that $H$ satisfies \ref{max-degree} and \ref{girth} deterministically. We will show that it satisfies the other three properties with positive probability. We  repeatedly use the following observation: for any fixed $T \subseteq V$ and $i \in [m],$ 
    \begin{equation} \label{eq:prob-edge}
        \Pr[T \subseteq B_i] = \binom{N-|T|}{s-|T|} / \binom{N}{s} \le \left( \frac{s}{N} \right)^{|T|}.
    \end{equation} 
    In particular, for any $v \in V, i \in [m],$ we have $\Pr[v \in B_i] = s/N.$ Let us verify that $H$ satisfies \ref{many-edges} with probability at least $3/4.$ Indeed, the expected number of hyperedges removed while removing short Berge cycles is at most
    \[ \sum_{\ell = 2}^g N^\ell m^\ell \left(\frac{s}{N}\right)^{2\ell} = \sum_{\ell=2}^g C^\ell s^{2\ell} < CN/32, \]
    where we used that $N$ is large enough. Indeed, the left hand side is derived by fixing every possible sequence of $\ell$ distinct vertices of $H$ (each one of them corresponding to an intersection vertex of two consecutive edges in a Berge cycle of length $\ell$) and $\ell$ indices in $[m]$ (corresponding to the $\ell$ edges of a Berge cycle), using \eqref{eq:prob-edge} and applying a union bound. Note that the term $\ell = 2$ also accounts for all pairs of edges which have at least two vertices in common and, in particular, for duplicate edges. The expected number of edges touching a vertex of degree more than $8Cs$ is at most 
    \[ \sum_{i=1}^{m} \sum_{v \in B_i} \Pr[d_{H_0}(v) \ge 8Cs + 1] \le \sum_{i=1}^m s \cdot \Pr[\Bin(m-1, s/N) \ge 8Cs] \le ms 2^{-7Cs}=(Cs)2^{-7Cs}N \le N/32, \] 
    where we used standard Chernoff bounds (see, e.g., Theorem 2.1 in \cite{JLR}). By Markov's inequality, with probability at least $7/8$ we removed at most $CN/4$ edges to remove short Berge cycles and with probability at least $7/8$ we removed at most $N/4 < CN/4$ edges incident to vertices of large degree. In that case, $CN \ge e(H) \ge CN/2,$ so with probability at least $3/4,$ $H$ satisfies \ref{many-edges}.
    
    Next, we show that \ref{no-dense-set} holds with probability at least $3/4.$ For convenience, we consider the sets $B_i, i \in [m]$ to also be given a random ordering $x^i_1, \dots, x^i_s$ of their elements. Consider a set $U$ of $u \leq \alpha N$ indices in $[m]$ and a graph $F$ on the vertex set $U$ with $\frac{4}{3}u$ edges. The graph $F$ will correspond to the subgraph $\calI'$ from the statement. We build an auxiliary graph $F'$ on the vertex set $V(F') = V(F) \times [s].$ For every edge $vw \in E(F),$ choose arbitrary $i, j \in [s]$ and add the edge $(v,i)(w,j)$ to $F'.$ Note that given $F,$ there are $s^{2 e(F)}$ ways to construct the graph $F'.$ For each $v \in V(F),$ the set $\{v\} \times [s]$ represents the $s$ elements of the ordered set $B_v$ and an edge $(v,i)(w,j) \in E(F')$ corresponds to the event that $x^v_i$ and $x^w_j$ are equal. Assume there is a subgraph $\calI' \subseteq \calI(H)$ such that $V(\calI')= U, $ $e(\calI') \ge \frac{4}{3}u$ and $\calI'$ has no sunflower cycle. Then, one can choose the indices for each edge in $F$ such that $F'$ is a forest. Indeed, a cycle in $F'$ corresponds to a sunflower cycle in $\calI'.$ Suppose $F'$ is a forest and let $C_1, C_2, \dots, C_t$ be its connected components. For $i \in [t],$ denote the event
    \[ A_i = \{ x^v_j \text{ are equal for all } (v, j) \in C_i \}. \] 
    Note that $\Pr[A_i \vert A_1 \land \dots \land A_{i-1}] \le \left(\frac{1}{N-s+1}\right)^{v(C_i) - 1}.$ Indeed, conditioning on the values of some elements of a set $B_v,$ the next element of the set is chosen uniformly at random among at least $N-s+1$ remaining choices. Hence, we have
    \[ \Pr[A_i, \forall i \in [t]] \le \prod_{i=1}^{t} (N-s+1)^{- (v(C_i) - 1)} = (N-s+1)^{-e(F)} < \left(\frac{2}{N}\right)^{4u/3}, \] where in the equality we used that $F'$ is a forest and $e(F') = e(F).$ By taking a union bound over all choices of $U, F, F',$ we get that the probability that $H$ violates \ref{no-dense-set} is at most
    
    \begin{align*}
        \sum_{u=1}^{\alpha N} \binom{m}{u} \binom{\binom{u}{2}}{4u/3} s^{8u/3} \left(\frac{2}{N}\right)^{4u/3} &\le \sum_{u=1}^{\alpha N} \left( \frac{eCN}{u} \right)^u \left ( \frac{3e}{8} u\right)^{4u/3} s^{8u/3} \left(\frac{2}{N}\right)^{4u/3} \\
        &< \sum_{u=1}^{\alpha N} \left(10 C s^{8/3}\right)^u \left( \frac{u}{N} \right)^{u/3}\\
        &\le \sum_{u=1}^{\alpha N} \left(10 C s^{8/3}\right)^u \alpha^{u/3} = \sum_{u=1}^{\alpha N} 10^{-u} < 1/4,
    \end{align*}
    as needed.
    
    Finally, we verify that $H$ satisfies \ref{no-dense-aux-graph-set} with probability at least $3/4.$ Indeed, for a fixed set $A,$ the probability that $|B_i \cap A| \ge 2$ is at most 
    ${|A| \choose 2} \left(\frac{s}{N}\right)^2 \leq s^2 \left(\frac{|A|}{N}\right)^2.$ Then, \ref{no-dense-aux-graph-set} fails with probability at most
    \begin{align*}
        &\sum_{u=1}^{\alpha N} {N \choose u} {m \choose 2u} \left( \frac{su}{N}\right)^{4u} 
        \le \sum_{u=1}^{\alpha N}  \left( \frac{eN}{u}\right)^u \left(\frac{eCN}{2u}\right)^{2u} \left( \frac{su}{N}\right)^{4u} \\
        &< \sum_{u=1}^{\alpha N} \left(\frac{u}{N}\right)^u \left( e^2Cs^2\right)^{2u} < \sum_{u=1}^{\alpha N} \alpha^u (10 C s^2)^{2u} < \sum_{u=1}^{\alpha N} 10^{-u} < 1/4,
    \end{align*}
    where in the second to last inequality we used that $\alpha = 10^{-6} C^{-3} s^{-8}.$
    Hence, $H$ satisfies \ref{many-edges}--\ref{no-dense-aux-graph-set} with probability at least $1/4,$ implying the existence of the desired hypergraph.
\end{proof}

We are now ready to present our construction of the host graph used to prove both parts of Theorem~\ref{thm:main}, and Theorem~\ref{thm:non-induced}. The construction differs for the three statements we are proving mainly because of the use of different gadgets, but for all constructions we use Lemma~\ref{lem:hypergraph}, just with different parameters. Let us remark that linear hypergraphs were also used for constructing Ramsey graphs in \cite{draganic2022size}, but in a different context.

To prove Theorem~\ref{thm:main}, we can evidently assume that $k$ is a large enough integer.

\noindent\textbf{Host graph.} We first describe the structure of our host graph construction which is the same for all three statements and then we give the parameters in each of the three cases. When the argument changes depending on which statement we are proving, we will denote $a$ equals $1, 2$ or $3$ depending on whether we are proving Theorem~\ref{thm:main}~\ref{induced-even}, Theorem~\ref{thm:main}~\ref{induced-odd} or Theorem~\ref{thm:non-induced}, respectively. Hence, we denote the host graphs used in the proofs of Theorem~\ref{thm:main}~\ref{induced-even}, Theorem~\ref{thm:main}~\ref{induced-odd} and Theorem~\ref{thm:non-induced} by $\Gamma_1, \Gamma_2$ and $\Gamma_3,$ respectively. Similarly, for $a \in [3],$ we denote by $F_a, s_a, C_a, H_a, \calI_a, G_a, L_a$ the different objects appearing in our proofs, some of which we define shortly. However, most of our arguments hold regardless of which of the three statements we are proving. In such cases, we omit the subscript $a$ from $\Gamma_a, G_a$, etc.

We let $F$ be a certain small graph  which does not depend on $n$ and we let $s = v(F).$ We choose an appropriate parameter $C$ and set $g = (Cs)^{20}$ and $N = 10^{100}k^2C^6s^{14}n.$ We apply Lemma~\ref{lem:hypergraph} to obtain an $s$-uniform hypergraph $H$ satisfying \ref{many-edges}--\ref{no-dense-aux-graph-set} for the given parameters.
Finally, we build a host graph $\Gamma$ on the vertex set $V(H)$ by arbitrarily placing an isomorphic copy of $F$ on each set of $s$ vertices forming a hyperedge in $H.$ The paremeters are chosen as follows.

\begin{itemize}
    \item Parameters for Theorem~\ref{thm:main}~\ref{induced-even} (induced even case, $a=1$)\\
    Let $F_1$ be the $k$-induced-Ramsey graph for the $6$-cycle provided by Lemma~\ref{lem:gadget even}. Recall that $s_1 = v(F_1) = O(k^6)$ and $e(F_1) = O(k^9).$ Set $C_1 = 10^{20}k$ and note that $e(\Gamma_1) = O(C_1Nk^9) = O(k^{102}) n,$ as required.
    \item Parameters for Theorem~\ref{thm:main}~\ref{induced-odd} (induced odd case, $a=2$)\\
    Let $F_2$ be the $k$-induced-Ramsey graph for the $5$-cycle provided by Lemma~\ref{lem:gadget odd}. Recall that $s_2 = v(F_2) = e^{O(k \log k)}.$ As above, set $C_2 = 10^{20}k$ and note that $e(\Gamma_2) = e^{O(k \log k)}n.$
    \item Parameters for Theorem~\ref{thm:non-induced} (non-induced case, $a=3$)\\
    Let $F_3$ be the complete graph on $s_3 = 2^k + 1$ vertices. Set $C_3 = 10^{20} \cdot k2^{3k}.$ We have $e(\Gamma_3) = O(CN \cdot 2^{2k}) = e^{O(k)}n.$
\end{itemize}

\noindent\textbf{Auxiliary graph.} Given a $k$-edge-colouring of the host graph $\Gamma_a$ with $a \in [3]$, we define an auxiliary $k$-edge-coloured graph $G_a$ on the same vertex set. 

\begin{itemize}
    \item Auxiliary graph for Theorem~\ref{thm:main} (induced case, $a \in \{1,2\}$)\\
    Recall that in each copy of $F$ that we placed in $\Gamma_a$, there exists a monochromatic induced cycle of length $L_1=6$ in the even case $a=1$, and of length $L_2=5$ in the odd case $a=2$. For each copy of $F_a$, choose one such cycle and choose a single pair of vertices at distance two on the cycle. Place an edge in $G_a$ joining the two vertices and colour it by the colour of the corresponding cycle.
    \item Auxiliary graph for Theorem~\ref{thm:non-induced} (non-induced case, $a=3$)\\
    By Lemma~\ref{lem:any-odd-cycle}, in each copy of $F_3$ placed in $\Gamma_3,$ there is a monochromatic odd cycle. For each copy of $F_3,$ choose one such cycle and let $L_3$ denote the most frequent length among these cycles in $\Gamma_3$. For each of the chosen cycles of length $L_3$, choose a single pair of vertices at distance $(L_3 - 1) / 2$ on the cycle and join them by an edge in $G_3$ whose colour is the same as the colour of the cycle.
\end{itemize}

We use $\calI_a = \calI_a(H_a)$ to denote the intersection graph of $H_a.$ For an edge $uv \in G_a,$ we use $h(uv)$ to denote the hyperedge in $H$ containing $u$ and $v$ and we use $h^{-1}$ to denote the inverse function of $h,$ so $h^{-1}(h(uv)) = uv.$

\section{Proof of Theorem~\ref{thm:main} and Theorem~\ref{thm:non-induced}}
In this section we present the proofs of our theorems. The proofs are split into four subsections. In the first subsection we define the relevant notions and reduce the problem of finding a monochromatic (induced) cycle of length $n$ in the host graph to the problem of finding a \emph{good} cycle (see Definition~\ref{def:good-cycle}) of approximate length in the auxiliary graph. The details (see \Cref{lem:cycle in host graph}) differ for the three different statements we are proving. In the remaining three subsections, we give a unified proof of the existence of the desired good cycle in the auxiliary graph. Each of these subsections deals with a different stage of building the good cycle, as outlined in Section~\ref{sec:outline}.

\subsection{Setting up}\label{sec:setup}
We start with a definition which is crucial in our proof. It describes paths in $G$ which we will use later to guarantee the existence of (induced) paths in the host graph.

\begin{defn} Given a path $P = (v_1, v_2, \dots, v_\ell)$ in $G$, we say that $P$ is \emph{good} if the following hold (see Figure
~\ref{fig:good path}).
\begin{itemize}
\item For every two edges $e,e'$ in $P$, $h(e)$ and $h(e')$ are disjoint unless $e$ and $e'$ are adjacent in $P$ (in which case $h(e) \cap h(e') = e\cap e'$).
\item There is no hyperedge in $H$, besides the hyperedges $h(v_iv_{i+1})$, which intersects $\cup_{i\in[\ell-1]}h(v_iv_{i+1})$ in more than one vertex. 
\end{itemize}
Otherwise, we say $P$ is \emph{bad}.
\end{defn}

\begin{figure}[h]
    \caption{Example of a good path. Edges of the path are represented by red segments and their corresponding hyperedges by full-lined ovals. By the first property of a good path, these hyperedges intersect only in the vertices on the path as depicted while by the second property, there is no other hyperedge intersecting the union of these hyperedges in more than one vertex, i.e. the dashed hyperedge on the picture does not exist.}
    \includegraphics[scale=0.8]{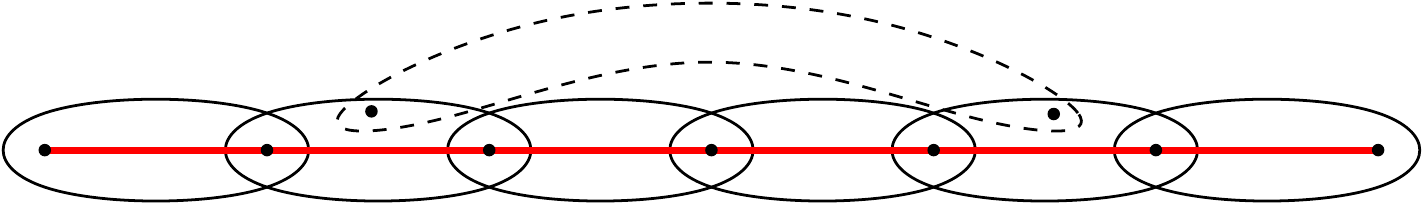}
\label{fig:good path}
\end{figure}

\begin{lemma}\label{lem:short is good}
Let $P = (v_1, v_2, \dots, v_\ell)$ be a path in $G$. If $\ell \le g,$ then $P$ is good.
\end{lemma}
\begin{proof}
    Denote $h_i = h(v_iv_{i+1}),$ for $1 \le i \le \ell-1.$ We prove the lemma by induction on $\ell.$ If $\ell = 2,$ the statement follows from the fact that $H$ is linear. Now consider the case $\ell = 3$ and suppose there is some $f \in E(H)$ other than $h_1, h_2$ containing two vertices in $h_1 \cup h_2.$ If $v_2 \in f,$ then $f$ intersects $h_1$ or $h_2$ in at least two vertices, contradicting linearity of $H.$ Else, let $x \in f \cap h_1, y \in f \cap h_2$ be such that $x, y, v_2$ are all distinct vertices. Then, $x, h_1, v_2, h_2, y, f$ is a Berge $3$-cycle in $H$, contradicting \ref{girth}, since the girth of $H$ is larger than $g \geq 4$.
    
    Now, assume $\ell \ge 4$ and $P$ is not good. By the induction hypothesis the paths $P_1 = (v_1, \dots, v_{\ell-1})$ and $P_2 = (v_2, \dots, v_\ell)$ are good. We then also have that $h_{\ell-1} \cap h_1 = \emptyset,$ as otherwise $h_{\ell-1}$ intersects $h_1$ and $h_{\ell-2}$, which contradicts that $P_1$ is good. Hence, $P$ satisfies the first point in the definition of a good path. So, there is some $f \in E(H) \setminus \{h_1, \dots, h_{\ell-1}\}$ intersecting $\cup_{i \in [\ell-1]}h_i$ in at least two vertices. Recalling that $P_1$ and $P_2$ are good, we have that $f$ intersects $h_1$ and $h_{\ell-1}$ in some vertices $x \neq v_2$ and $y \neq v_{\ell-1}.$ But now $x, h_1, v_2, h_2, \dots, v_{\ell-1}, h_{\ell-1}, y, f$ is a Berge cycle of length $\ell \le g$ in $H$, contradicting \ref{girth}.
\end{proof}

\begin{observation}\label{obs:ruin}
Let $P$ be a good path in $G$, let $v$ be an endpoint of $P$, and let $u$ be a vertex outside of $P$. If $P+uv$ is a bad path in $G$, then there is a hyperedge $h\in V(H)$ (different from $h(vu)$) and an edge $e$ on $P$ at distance at least $2$ from $v$, such that both $\{h,h(vu)\},\{h,h(e)\}\in E(\mathcal I).$ We say that $h$ \emph{ruins} $P + uv.$
\end{observation}

\begin{proof}
    Suppose $P+uv$ is a bad path. By definition, since $P$ is a good path, there are only two cases:
    \begin{itemize}
        \item $h(uv)$ intersects $h(e)$ for some edge $e\in P$, different from the last edge in $P$, which contains $v$. But this is not possible, since then $h(uv)$ intersects two hyperedges corresponding to edges on the path, which would imply that $P$ is bad, a contradiction.
        \item There is another hyperedge $h\in E(H)$, which intersects $h(e)$ and $h(e')$, for two distinct edges $e,e'\in P+uv$. One of those must be the edge $uv$, since otherwise we again get a contradiction with $P$ being good. The other one cannot be one of the last two edges of $P$, since otherwise the path consisting of those two edges and $uv$ would not be a good path, a contradiction with Lemma~\ref{lem:short is good}.
    \end{itemize}
    \vspace{-6mm}
\end{proof}

\begin{defn}[Good tree]
    A rooted tree in $G$ is called \emph{good} if its every path containing the root is good, otherwise the tree is \emph{bad}.
\end{defn}

\begin{defn}[Good cycle] \label{def:good-cycle}
    A cycle $Q$ in $G$ is said to be \emph{good} if there is a collection $\calP$ of good subpaths of $Q$ such that every pair of edges of $Q$ appears in a member of $\calP$. 
\end{defn}

We are now ready to prove that finding a monochromatic good cycle of certain approximate length in the auxiliary graph $G$ yields the desired (induced) cycle of length $n$ in the host graph $\Gamma.$
\begin{lemma}\label{lem:cycle in host graph}
Let $Q$ be a monochromatic good cycle of length $\ell$ in $G_a$. Then,
\begin{enumerate}[label=\arabic*)]
    \item If $a = 1,$ then there is a monochromatic induced cycle in $\Gamma_a$ of length $t$ for each even $t \in [2\ell, 4\ell].$
    \item If $a = 2,$ then there is a monochromatic induced cycle in $\Gamma_a$ of length $t$ for each $t \in [2\ell, 3\ell]. $
    \item \label{lem:cycle-in-host-non-induced} If $a = 3,$ then there is a monochromatic cycle in $\Gamma_a$ of length $t$ for each $t \in [\frac{L_3-1}{2} \ell, \frac{L_3+1}{2}\ell].$
\end{enumerate}
\end{lemma}

\begin{proof}
    Let $Q = (v_0, v_1, \dots, v_{\ell-1})$ and suppose the desired cycle is of length $t.$  Throughout the proof we will consider the indices of vertices $v_i$ to be taken modulo $\ell.$  Note that the hyperedges $h(v_iv_{i+1})$ for  $0 \le i < \ell$ are all distinct because $G$ contains exactly one edge for every hyperedge in $H.$ Let us call the colour of $Q$ red. Recall that each edge $v_iv_{i+1}\in E(Q)\subseteq E(G)$ corresponds to a pair of vertices in a red cycle of length $6$ when $G=G_1$, of length $5$ when $G=G_2$, and length $L_3\leq 2^k+1$ when $G=G_3,$ their distance on the cycle being $2$ in the first two cases, while their distance is $(L_3-1)/2$ in the third case. By choosing the appropriate path on the cycle for each pair $v_iv_{i+1},$ in each of the three cases we obtain a red closed walk $Q' = u_0u_1\dots u_{t-1}$ of length $t,$ where $u_0 = v_0$. It remains to show that $Q'$ is indeed a cycle, that is, no vertices are repeated, and that it is induced in $G_a$ when $a \in \{1, 2\}.$ Denote by $U_j \subseteq V(Q')$ the set of vertices in $\Gamma_a$ forming the chosen paths between $v_j$ and $v_{j+1}$ and note that $U_j \subseteq h(v_jv_{j+1}).$ 
    
    Suppose first that $Q'$ is not a cycle so there are indices $0 \le j < j' < t$ such that $u_j = u_{j'}.$ By definition, $u_j$ comes from a path between two vertices $v_i,v_{i+1}$, and $u_{j'}$ from a path between two vertices $v_{i'},v_{i'+1},$ where $i \neq i',$ because we connected consecutive $v_i$'s by paths. Furthemore, the edges $v_iv_{i+1}, v_{i'}v_{i'+1}$ cannot be consecutive on the cycle. Indeed, suppose, without loss of generality, that $i' = i+1.$ However, the corresponding hyperedges $h(v_iv_{i+1}), h(v_{i'}v_{i'+1})$ intersect only in $v_{i+1},$ which is an endpoint of these paths, a contradiction. Finally, if the edges $v_iv_{i+1}, v_{i'}v_{i'+1}$ are not consecutive on $Q,$ it implies that the hyperedges $h_1 = h(v_iv_{i+1})$ and $h_2 = h(v_{i'}v_{i'+1})$ are not disjoint. Since $Q$ is a good cycle, there is a good subpath of $Q$ containing $h_1$ and $h_2,$ a contradiction. This shows that $Q'$ is indeed a red cycle and in particular, completes the proof of \ref{lem:cycle-in-host-non-induced}.
    
    Now, let $a \in \{1, 2\}$ and for the sake of contradiction, assume there is an edge $xy \in E(\Gamma_a)$ between two nonconsecutive vertices in $Q'$ and let $f \in E(H)$ be the unique hyperedge containing both $x$ and $y.$ If $f = h(v_jv_{j+1})$ for some $0 \le j \le \ell-1,$ this is a contradiction since the $5$ or $6$-cycle we found in $f$ is induced. Otherwise, $f$ is not equal to any $h(v_jv_{j+1}).$ Let $j \ne j'$ be indices such that $x \in U_j, y \in U_{j'}$ and let $P \in \calP$ be a good path containing $v_jv_{j+1}, v_{j'}v_{j'+1}.$ Then $f$ ruins $P,$ a contradiction.
\end{proof}

Recall that $L_1 = 6, L_2 = 5,$ hence for any $a \in \{1, 2, 3\},$ finding a monochromatic good cycle in $G_a$ of length between $\frac{2n}{L_a + 1}$ and $\frac{2n}{L_a - 1}$ yields the desired monochromatic cycle in $\Gamma_a.$

We will also need the following theorem by Krivelevich \cite{krivelevich2019expanders}, which states that locally sparse graphs contain a large expanding subgraph. For our application of this result, we will also require that the found subgraph has large average degree. The theorem in \cite{krivelevich2019expanders} is not stated with the average degree condition, but it can be easily extracted from its proof. For completeness, we include the proof in the appendix.

\begin{defn}
    Let $G = (V, E)$ be a graph on $n$ vertices, and let $\gamma > 0$. The graph $G$ is a $\gamma$-expander if $|N_G(U)| \ge \gamma|U|$ for every vertex subset $U \subseteq V$ with $|U| \le n/2.$
\end{defn}

\begin{restatable}[\cite{krivelevich2019expanders}]{theorem}{thmmichaelrestate} \label{t-sparse}
Let $c_1>c_2>1$, $0<\beta<1$, $\Delta>0$. Let $G=(V,E)$ be a graph on  $n$ vertices, satisfying:
\begin{enumerate}
\item $\frac{|E|}{|V|}\ge c_1$\,;
\item every vertex subset $U\subset V$ of size $|U|\le \beta n$ spans fewer than $c_2|U|$ edges;
\item $\Delta(G)\le \Delta$.
\end{enumerate}
Then $G$ contains an induced subgraph $G^*=(V^*,E^*)$ on at least $\beta n$ vertices which is a $\gamma$-expander, for $\gamma=\frac{c_1-c_2}{2\Delta\cdot\left\lceil\log_2\frac{1}{\beta}\right\rceil}$, with $|E^*| / |V^*| \ge \frac{c_1 + c_2}{2}.$
\end{restatable}

In what follows, we denote by $G'$ the graph obtained from $G$ as follows. Consider the subgraph $G_{red}$ of $G$ consisting of the edges of the densest color class.
In the induced case ($a \in \{1,2\}$), since $G$ has at least $\frac{CN}{2}$ edges (one for each gadget), we get that $G_{red}$ has at least $\frac{CN}{2k}\geq 10^{19} N$ edges.
In the non-induced case ($a=3$), since $G$ has at least $\frac{CN}{2\cdot 2^{k}}$ edges (one for each gadget which gives the most common monochromatic cycle), we conclude that $G_{red}$ has at least $\frac{CN}{2k2^k}\geq 10^{19}2^{2k} N$ edges.
Furthermore, by the last property in Lemma~\ref{lem:hypergraph}, every set $S\subset V(G)$ of size at most $\alpha N$, spans at most $2|S|$ edges. Hence, the conditions of Theorem~\ref{t-sparse} are satisfied for the graph $G_{red}$ with 
$\beta=\alpha=10^{-6}C^{-3}s^{-8}$, $\Delta=8Cs\geq \Delta(H)\geq \Delta(G_{red})$, $c_2=2$ and $c_1=10^{19}$ in the induced case ($a \in \{1,2\}$), while $c_1=10^{19}2^{2k}$ in the non-induced case ($a=3$). Thus we obtain a subgraph $G'_{red}$ of $G_{red}$ on at least $\beta n$ vertices which is $\gamma n$-expanding with average degree $d$, where $d\geq 10^{18}$ in the induced case ($a \in \{1,2\}$) and $d\geq 10^{18}2^{2k}$ in the non-induced case ($a=3$), and
$$\gamma=\frac{c_1-c_2}{2\Delta\log_2\frac{1}{\beta}}\geq \frac{10^{19}-2}{16Cs\log_2(10^6C^3s^8)}\geq \frac{1}{k^2Cs}$$
where in the last inequality we used $\log_2(Cs)= \log_2 (e^{O(k\log k)})= O(k \log k)$. Now let $G'$ be a maximal subgraph of $G_{red}'$ with minimum degree $\delta(G') \ge d/2 > 10^{17}$ in the induced case ($a \in \{1,2\}$) and $\delta(G') \geq 10^{17}2^{2k}$ in the non-induced case ($a=3$), so in each case $\delta(G')\geq 10^{15}L^2$ as we have that $L_1=6$, $L_2=5$ and $L_3\leq 2^k+1$. Note that $G'$ has at least $\alpha N$ vertices. Indeed, if that was not the case then we would have a subgraph of $G_{red}$ on less than $\alpha N$ vertices which contains at least $|V(G_{red})|d/4>2|V(G_{red})|$ pairs of vertices which are in an edge of $H$, a contradiciton with the last property of Lemma~\ref{lem:hypergraph}.

\subsection{Finding a long good path}
Now everything is set up to start building a good cycle in $G'$. This section is dedicated to finding a long good path in $G'$, which is the statement of the following lemma. In this lemma $L$ is one of $L_1, L_2$ or $L_3$, however, our arguments hold regardless of which of the three cases we are proving and therefore we omit the subscript.

\begin{lemma}
    There is a good path of length $\frac{2}{L}n$ in $G'.$
\end{lemma}
\begin{proof}
To prove this lemma, we run a modification of the DFS graph search algorithm on $G'$, and by analyzing it we show the existence of the required path.
    We maintain four sets $P, S_1, S_2, U \subseteq V(G')$, together with a subgraph $F\subseteq \cI$ containing no sunflower cycle. In the beginning we set $P = S_1 = S_2 =  \emptyset$, $U=V(G'),$ and we let $F$ be the empty graph.
    The vertices in $P$ will always form a good path in $P$. As we explore the vertices of $G'$, we will either increase $P$, $S_1$ or $S_2$, while $U$ shrinks. 
    We add vertices from $P$ to $S_1$ when they do not have neighbours in the unexplored vertices, while to $S_2$ we will add vertices when we add new edges to the auxiliary graph $F\subseteq \cI$. The latter case will be a consequence of the existence of hyperedges which ruin the goodness of the current path.
    If $S_1$ grows too large, we will show the existence of a small set of vertices in $G'$ spanning too many edges in $G'\subset G$, which will contradict the local density condition of $G$. On the other hand, if $S_2$ grows too large, then we will show that the auxiliary graph $F\subseteq \cI$ has too many edges, contradicting the local density condition in $\cI$. As a consequence, this will imply that $P$ is large enough at some point, finishing the proof.

    The vertices in $P$ are kept in a stack. We run the following algorithm in rounds until $U=\emptyset$:
    
\noindent    \textbf{Start of round}
    \begin{enumerate}
        \item If ${P}=\emptyset$, remove an arbitrary vertex from ${U}$ and push it to ${P}.$ Let $v$ be the last vertex in ${P}.$
        \item If $v$ has no neighbours in ${U},$ pop $v$ from ${P}$ and add it to ${S_1}$. 
        \item\label{step:add to P} Else, let $u$ be any neighbour of $v$ in ${U}$, remove it from $ U$ and add it to $P$. 
        If the new path $P$ is good, add\footnote{\label{foot} We argue later why $h_1$ is not already in $V(F).$} $h(uv)$ to the vertex set of $F$.
        \item\label{step:remove from P} If $P$ is not good, do the following. Let $v'$ be the vertex in $P$ before $v$. Consider $h_1:=h(vu)$ as well as exactly one hyperedge $h_2$ which ruins ${P}+vu$ because it intersects $h_1$ and $h(e)$ for some edge $e$ on the path $P$ (given by Observation \ref{obs:ruin}).
        If $h_2$ is not in $V(F),$ add it to $V(F)$ and add the edge $\{h_2, h(e)\}$ to $E(F).$ Add\footnotemark[\value{footnote}] $h_1$ to $V(F)$ and add edges $\{h_1, h_2\}$ and $\{h_1, h(vv')\}$ to $E(F).$ Finally, for each vertex in the edge $h^{-1}(h_2)$ which is in $U$, remove it from $U$ and add it to $S_2$. Pop $u$ from $P$ and add it to $S_2$. (See Figure~\ref{fig:step-4} for an illustration.)
    \end{enumerate}
  \textbf{End of round}
  
  \begin{figure}[h]
    \caption{Illustration of Step 4 of the algorithm. Figure a) on the right depicts the case when $h_2$ is not already in $V(F)$ and then two vertices ($h_1$ and $h_2)$ are added to $V(F)$ and three edges ($\{h(vv'), h_1\}, \{h_1, h_2\}$ and $\{h_2, h(e)\}$) are added to $E(F).$ Figure b) depicts the case when $h_2$ is already in $V(F)$ and then one vertex ($h_1$) is added to $V(F)$ while two edges ($\{h(vv'), h_1\}$ and $\{h_1, h_2\}$) are added to $E(F)$.} \label{fig:step-4}
    \includegraphics[scale=0.89]{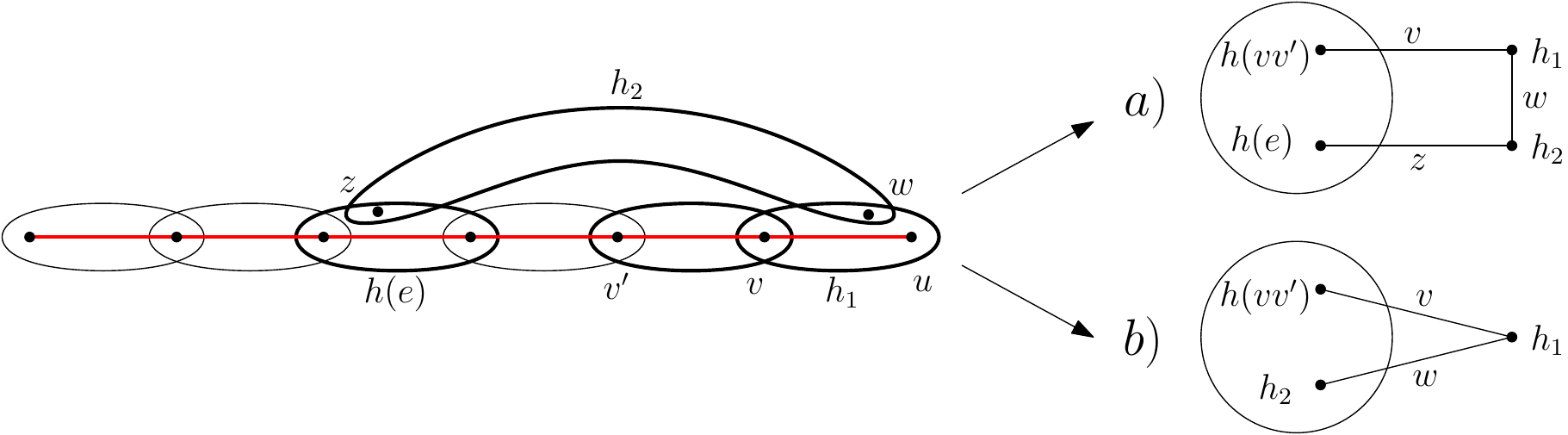}
\end{figure}
  
  We first show some simple invariants which are maintained during the execution of our algorithm.
 \begin{claim}
The following hold at the end of each round of our modified DFS algorithm:
\begin{enumerate}[label= (\Alph*)]
\item\label{p:goodpath} The vertices in $P$ form a good path in $G$.
\item\label{p:not in U} When a hyperedge $h(xy)$ is added to $V(F)$ then both $x$ and $y$ are not in $U$ from that point until the end of the algorithm.
\item\label{p:no edge} No vertex in $S_1$ is adjacent to $U$.
\item\label{p:size of F} At each point of the algorithm $F$ contains at least $|S_2|/3$ vertices and at most $|P|+|S_1|+2|S_2|$ vertices.
\item\label{p:density} At each point of the algorithm $F$ contains at least $\frac{3}{2}(v(F)-|P|-|S_1|)$ edges and no sunflower cycle.
\end{enumerate}
 \end{claim}
 \begin{proof}
  $ $\newline
     \begin{itemize}
     \item Property~\ref{p:goodpath} is trivially satisfied, since in Step~\ref{step:remove from P} we remove the vertex which made the path bad. Property~\ref{p:not in U} also holds since a hyperedge $h(xy)$ can be added to $V(F)$ only in Steps~\ref{step:add to P} or \ref{step:remove from P} and in both cases both its endpoints end up in $P \cup S_1 \cup S_2$. Furthermore, a vertex which is added to $P \cup S_1 \cup S_2$ at any point of the algorithm, stays in that set throughout. Also \ref{p:no edge} is easy to check, as a vertex is added to $S_1$ precisely when it has no neighbours to $U$, and since $U$ can only shrink, the claim holds. 
        
    \item For properties \ref{p:size of F} and \ref{p:density}, we first prove that the hyperedge $h_1=h(uv)$ from Steps~\ref{step:add to P}~and~\ref{step:remove from P} was not contained in $V(F)$ before the start of this round. Suppose for contradiction that it was. First, if it was added to $V(F)$ in some previous round, then by~\ref{p:not in U}, none of its vertices are in $U$ from the point $h_1$ was added to $V(F)$. Further, since we only add vertices to $V(F)$ in at most one of the Steps~\ref{step:add to P}~and~\ref{step:remove from P}, then $h_1$ could not have been added to $V(F)$ earlier in the same round.
    
    \item Now, we show \ref{p:size of F}. For the lower bound, note that $S_2$ only increases in Step 4, where each time at most $3$ vertices are added to $S_2$ and at least one vertex $h(uv)$ is added to $V(F).$ For the upper bound, notice that there are two ways in which we add vertices to $V(F)$. The first is when in Step~\ref{step:add to P} we add a vertex to $P$ (which either stays in $P$ or is moved to $S_1$). The second is when in Step~\ref{step:remove from P} we add at most two vertices to $V(F)$, while at the same time we add at least one vertex to $S_2$. This finishes part \ref{p:size of F} since $V(F)$ never decreases and vertices which are put in $S_2$ stay in $S_2$ throughout.
    
    \item To show the quantitative bound of \ref{p:density}, we prove that every time we reach Step~\ref{step:remove from P} and $P$ is bad, we either add one vertex and two edges to $F$, or we add two vertices and three edges. This would be enough, as in that step we added all but at most $|P|+|S_1|$ of the $v(F)$ vertices to $V(F)$. The remaining ones are added in Step~\ref{step:add to P}. Each time we enter Step~\ref{step:remove from P} and $P$ is not good, we distinguish two cases. In the first case, $h_2$ is already a vertex in $F$, hence to $F$ we only add the vertex $h_1$, together with its two edges to $h_2$ and $h(vv')$. In the second case, we also add $h_2$ to $F$, and together with it the edge $(h_2,h(e))$, which completes the first part of the proof.
    
    \item What is left is to argue that $F$ contains no sunflower cycle. Look at $F$ at an arbitrary point of the algorithm, and assume that it has no sunflower cycle; we now show that after an additional round no sunflower cycle is created.
    
    Trivially, when we add an isolated vertex to $V(F)$ in Step~\ref{step:add to P}, no sunflower cycle is created. 
    Now we consider Step~\ref{step:remove from P}, and suppose that the path $P$ is not good when we enter this step, as otherwise we are done.
    It is enough to show that the edges in $\cI$ touching the vertex $h_1=h(uv)$ which are added to $F$ have different labels, as well as that the two edges touching $h_2$ have different labels if $h_2$ is also added to $V(F)$ in this step. Thus, no sunflower cycle can be closed by these new edges, which would complete the proof. 
    
    First we consider $h(uv)$. The label of $(h_{vv'},h(uv))$ is $v$, so if $(h(uv),h_2)$ also had the label $v$, then $h_2$ also contains $v$ and hence it intersects $h(vv')$. By \Cref{obs:ruin}, $h_2$ also intersects $h(e)$ for some edge $e$ of $P$ at distance at least $2$ from $u$. Hence $h_2$ intersects the corresponding hyperedges of two distinct edges in $P-u$, which contradicts the fact that $P-u$ is good.
    
    Now, let us argue that if $h_2$ is added to $V(F)$, then the two edges added with it have different labels. Let the label of $(h_1,h_2)$ be $w$. First, note that $w\neq v$ by the same argument as in the previous passage. Now, if the label of the edge $(h_2,h(e))$ is also $w$, then the hyperedges $h(e)$ and $h_1$ intersect in $w$, which leads to a contradiction as the only hyperedge on the path which $h_1$ intersects is $h(vv')$, and they only share vertex $v$.
    
    \end{itemize}
    \vspace{-6mm}
 \end{proof}
 
Recall that $\alpha=10^{-6}C^{-3}s^{-8}$, $N = 10^{100}k^2C^6s^{14}n$ and $G'$ has at least $\alpha N$ vertices. Suppose for the sake of contradiction that at each point of the execution of the algorithm, $P$ is always of size $|P|\leq 2n/L < \alpha N / 10^6$.
Since the algorithm terminates when $U=\emptyset$, and $|P|\leq 2n/L$, at some point of the algorithm we have that either $|S_1|\ge\alpha N/100$ or $|S_2|\ge \alpha N/3$. We distinguish two cases depending on which of these two occurs first.

Suppose $|S_1|\ge\alpha N/100$ occurs first, and look at the moment when this happens. Then we have $|S_2|<\alpha N/3$. By property $\ref{p:no edge}$, all edges in $G'$ which touch $S_1$ are contained in $G[P\cup S_1\cup S_2]$. Thus we have a subgraph of $G$ on $|P\cup S_1\cup S_2|\leq \alpha N$ vertices and at least $\delta(G') |S_1| / 2\geq 10^{16}|S_1|>2|P\cup S_1\cup S_2|$ edges, contradicting \ref{no-dense-aux-graph-set} from Lemma~\ref{lem:hypergraph}.

On the other hand, when $|S_2|\geq \alpha N/3$ occurs first, we have that $|S_1|<\alpha N/100$. In that case, by \ref{p:size of F} we obtain a subgraph $F$ of $\cI$ with $|S_2| / 3 \le v(F) \le |P| + |S_1| + 2|S_2| \le \alpha N$ and using \ref{p:density} it has at least $\frac{3}{2}(v(F)-|P|-|S_1|)\geq \frac{4}{3}|v(F)|$ edges and no sunflower cycle, thus contradicting \ref{no-dense-set}.
\end{proof}

\subsection{Growing the initial trees}
Now that we have constructed a good path $P$ of length $2n/L$, we proceed with the second stage of the proof. Namely, we pick a vertex in the middle of $P$ to be the root, and we show how to attach (inside of $G'$) two large trees to the endpoints of $P$ so that the resulting rooted tree is still good.
Before we show the details, let us briefly outline this part of the proof.

Starting with one endpoint of the path, we try to attach to it a set $X_1$ of $10^4$ leaves so that the obtained tree $T$ is also good. Then we look at the other endpoint of $P$ and try to add a set $X_2$ of $10^4|X_1|$ leaves to that vertex so that $T+X_2$ is also a good tree. Then we again try to add $10^4|X_2|$ leaves adjacent to $X_1$ in that obtained tree and so on. We essentially try to make the tree larger (while preserving the goodness condition) by alternatively adding new leaves to one of the "sides" of the tree, each time adding a factor of $10^4$ more leaves to the other side then in the previous turn. Now, either we built a large good tree, in which case we are happy, or one of the following two cases happens. First, if the current set of leaves $X_1$ does not have at least $5\cdot 10^4|X_2|$ neighbours, then we will remove both attached trees from the current good tree to obtain a good path. 
Second, if $X_1$ does have many neighbours, but less than $10^4|X_2|$ can be attached to give a \textbf{good} tree, we again remove both attached trees from the current good tree to obtain a good path, and we start over.

Now, if any of those two scenarios happens too often, then we get the same type of contradiction as in the previous section. In particular, if the first scenario happens too often, it will be not hard to see that there exists a set $S$ of vertices in $G'$ which has a small external neighbourhood, which will give a contradiction with the local density condition of $G\supseteq G'$. If the second scenario happens too often, then the point is that each of the at least $4\cdot 10^4|X_2|$ leaves $v$ which can be attached to $X_1$ gives a bad tree, hence in the hypergraph $H$ we will have a hyperedge which ruins the path which contains the root and $v$. Using this fact we will be able to build an auxiliary graph $F\subseteq \cI$ which has too many edges and contains no sunflower cycle. Indeed, roughly speaking, for many vertices $v$ which can be attached to $X_1$ and give a bad tree, we will add at least three edges for every two vertices which we add to $F$ (analogously to Fig.~\ref{fig:step-4}). 

The technical details are slightly more delicate here than in the previous section, as in the second scenario we have to carefully pick which vertices and edges from $\cI$ are chosen to be added to the auxiliary subgraph $F\subseteq \cI$, which is supposed to give the required contradiction using the local density property of $\cI$. Let us also remark that the alternating sizes of the leaf sets are important to reach the mentioned contradictions regarding local density.

Let us now describe our algorithm in full detail. We perform the following procedure.

Let $v_1$ and $v_2$ be the endpoints of $P$. Let $S_1=\emptyset$, $S_2=\emptyset$ and $U=V(G')-V(P)$. Let $T_1$ be the tree consisting only of vertex $v_1$, and we let $X_1$ be the leaf set of $T_1$, so in the beginning let $X_1=\{v_1\}$.
Similarly $T_2=v_2$ and $X_2=\{v_2\}$. We always denote $T=T_1\cup P\cup T_2$.
We want to extend $T$, and we proceed in rounds, until $|T| \ge \alpha N/10^9$ or $|S_1| \ge n/L^2$ or $|S_2| \ge 100n$ occurs at the end of some round. We do so by alternately increasing $T_1$ and $T_2$, in each round making the leaf set of one of the trees by a factor of $10^4$ larger than that of the other tree. In particular, at every point of the algorithm we will maintain $\frac{|X_1|}{|X_2|}\in\{1,10^4, 10^{-4}\}$.
We further let $F$ be a subgraph of $\cI$, consisting only of isolated vertices $h(e)$ for each edge $e$ induced by $P$ in $G'$. The sets $S_1,S_2$ and the graph $F$ will serve a similar purpose as in the previous subsection. The difference is that now we do not have only one vertex which is in some sense bad and is then discarded, but a large part of a leaf set $X_1$ or $X_2$ (of one of the trees we consider) can be bad.

 In the algorithm given below, the neighbourhood $N_{B}(A)$ is defined as the neighbourhood of the vertex set $A$ in the vertex set $B$, in the graph $G'.$ The algorithm is carried out in rounds, as follows.

\noindent  \textbf{Start of round}
  
  W.l.o.g.\ assume that $|X_1|\leq |X_2|$ (otherwise rename them, and rename $T_1$ and $T_2$ analogously).
    \begin{enumerate}
        \item\label{lab:no expansion} If $|N_{U}(X_1)|\leq 5 \cdot 10^4|X_2|$, then remove $v_1$ from $P$, and add $V(T_1\cup T_2)-v_2$ to $S_1$. Denote by $u$ the vertex that was adjacent to $v_1$ in $P$, and let $X_1=\{u\}$ and let $T_1=u$. Further, let $X_2=\{v_2\}$ and $T_2=v_2$. Start the next round.
        \item\label{lab:expansion} Else, we have $|N_{U}(X_1)|> 5\cdot 10^4|X_2|$. To each vertex $v$ in $N_{U}(X_1)$ assign precisely one edge $vv'$ which connects it to the leaf set $X_1$ of the tree $T_1$. Let $v''$ be the ancestor of $v'$ in $T$. Let $B \subseteq N_U(X_1)$ be a set of $5 \cdot 10^4|X_2|$ of the found vertices $v$, with some arbitrary ordering $v_1,v_2,\ldots$
        \item\label{lab:deletions} Do the following for each $i = 1, \dots, 5 \cdot 10^4 |X_2|.$ If $v_i \not \in B$ (i.e., was already removed from $B$), continue to $i+1.$ Else, if $T+v_i'v_i$ is not good, let $h_i$ be the hyperedge which ruins it, and let $e_i$ be the edge in $T$ on a path which contains the root and $v_i'$, such that $h_i$ intersects $h(e_i)$ (note that by \Cref{obs:ruin} we have $e_i\neq v_i'v_i''$ and $e_i\neq v_iv_i'$). For each $j>i$, remove all vertices $v_j$ from $B$ which are contained in the two-element set $h^{-1}(h_i)$.
        \item\label{lab:add to T} If after the previous step there are at least $10^4|X_2|$ vertices $v$ in $B$ for which $T+vv'$ is good, then update $T_1$ by adding $10^4|X_2|$ such edges $vv'$, and let $X_1$ be the set of those $10^4|X_2|$ vertices which are leaves in the new tree $T_1$. Remove from $U$ all vertices $v$ which are added to $X_1$.
        
        \item\label{lab:bad tree} Otherwise, there is a set $B'$ of $10^4|X_2|$ vertices $v$ in $B$ for which $T+vv'$ is not good (for each such vertex we removed at most $2$ vertices in $B$, and at most $10^4|X_2|$ vertices in $B$ give a good tree, and we had $|B|\geq 5 \cdot 10^4|X_2|$). Proceed as follows.
        \begin{itemize}
            \item Add all vertices in $T-V(P)$ to $S_2$, and for each edge $e\in E(T) \setminus E(P)$ add\footnote{By property~\ref{i:not in U} below, it follows that each such $h(e)$ is not already in $V(F)$ since one of the endpoints of $e$ was in $U$ at the start of the round.} $h(e)$ to $V(F)$.
            \end{itemize}
            For each $v_i\in B'$, following the order inherited from $B$ do:  
            \begin{itemize}
            \item Move $v_i$ from $U$ to $S_2$, and add $h(v_iv_i')$ to $V(F)$ (later we argue why $h(v_iv_i')$ is not already there).
            \item Add $h_i$ to $V(F)$ (if it is
            not already there), and add both vertices in $h^{-1}(h_i)$ to $S_2$ (note that one or both of them could already be in $S_2$) and remove them from $U$ if they are there.
            \item Add the edges $\{h(v_i'v_i),h(v_i''v_i')\}$ and  $\{h(v_i'v_i),h_i\}$ to $F$, together with the edge $\{h_i,h(e_i)\}$ if $h_i$ was just added to $V(F)$.
        \end{itemize}

    \end{enumerate}
  \textbf{End of round}

\begin{claim}\label{cl:properties} 
The following hold at the end of each round of the above algorithm:
\begin{enumerate}[label= (\Alph*)]
\item\label{i:goodness} $T$ is a good tree in $G$.
\item\label{i:not in U} When a hyperedge $h(xy)$ is added to $V(F)$ then both $x$ and $y$ are not in $U$ from that point until the end of the algorithm.
\item\label{i:bad expansion} There is a subset $S\subseteq S_1$ of size $|S|\geq |S_1|/10^5$ with $|N_U(S)|\leq 5\cdot 10^8|S|$
\item\label{i:size of F} At the end of each round, $F$ contains at least $|S_2|/4$ vertices and at most $n+2|S_2|$ vertices.
\item\label{i:density} At each point of the algorithm $F$ contains at least $1.45(v(F)-n)$ edges and no sunflower cycle.
\item\label{i:length} The size of $P$ is always at least $|P|\geq \frac{2}{L+1}n$.
\end{enumerate}
 \end{claim}
 \begin{proof}
    In the beginning we have $|X_1| = |X_2| = |T_1| = |T_2| = 1,$ whereas afterwards we have $|X_2| = 10^4 |X_1|$ and $|T_1 \cup T_2| < \sum_{i=0}^\infty 10^{-4i} |X_2| < 1.1 |X_2|.$ Hence, it always holds that $|X_2| \ge \frac{1}{2}|T_1\cup T_2|$ and $|X_1|\geq |T_1\cup T_2|/10^5$.
    We proceed with proving the individual claims.
     \begin{itemize}
     \item For \ref{i:goodness}, observe that we only possibly add edges to $T$ in Step~\ref{lab:add to T} if $T$ stays good. Indeed, by definition of a good tree, if $T + vv'$ is good for each vertex $v$ that is added, than for the new tree $T$ obtained by adding all of these vertices, $T$ is also good. To see \ref{i:not in U}, note that vertices are added to $V(F)$ only at the beginning of the algorithm and in Step~\ref{lab:bad tree}. In both cases, once a hyperedge $h(xy)$ is added to $V(F),$ both $x$ and $y$ end up in $P \cup S_1 \cup S_2$ and stay in this set throughout. For \ref{i:bad expansion}, notice that vertices are added to $S_1$ only in Step \ref{lab:no expansion} when the neighbourhood of $X_1$ in $U$ is at most $5\cdot 10^4|X_2|\leq 5\cdot 10^8|X_1|$, and then we add to $S_1$ all vertices in $T_1\cup T_2-v_2$. Since $|X_1|\geq |T_1\cup T_2-v_2|/10^5$ and throughout the algorithm $U$ only shrinks, the claim holds by taking $S$ to be the union of the sets $X_1$ which are added to $S_1$ during the algorithm.
     \item  As in the proof in the previous subsection, for properties \ref{i:size of F} and \ref{i:density}, we first prove that every hyperedge $h(v_i'v_i)$ from Step~\ref{lab:bad tree}, added in any round, was not contained in $V(F)$ already. First, it was not added to $V(F)$ in the very beginning, since then $v_i'v_i$ would need to be part of the initial path, but since $v_i'$ was in $U$ when starting the observed round, this is a contradiction.
     
     Furthermore, $h(v_iv_i')$ was not added to $V(F)$ in a previous round, since in this case both $v_i$ and $v_i'$ would have been added to $S_2$, which gives the same contradiction. Finally, we show that $h(v_iv_i')$ could not have been added to $V(F)$ in an earlier stage of Step~\ref{lab:bad tree} of the same round. Indeed, $v_iv_i'$ is certainly not equal to some $v_jv_j'$ with $j<i$, and also $h_j\neq h(v_iv_i')$ since in Step~\ref{lab:deletions} we deleted vertices in $h^{-1}(h_j)$ from $B$ which come later in the ordering.
    
    \item To show \ref{i:size of F}, recall that we add vertices to $V(F)$ in one of the two following cases. First, we add at most $n$ hyperedges/vertices to $V(F)$ at the beginning of the algorithm, one for each edge on the initial path $P$.
    
    Second, in Step \ref{lab:bad tree}, to $V(F)$ we added at least $|B'|=10^4|X_2|$ vertices. On the other hand, we added to $S_2$ at most $3$ vertices for each of those $|B'|$ vertices (one for each $h(v_i'v_i)$ and at most two for $h_i$), and at most one for each edge in $T_1\cup T_2$ where $|T_1\cup T_2| \leq 2|X_2|<|B'|/100$. This implies that to $V(F)$ we added at least $|S_2|/4$ vertices, which gives the required lower bound. For the upper bound, note that every time we add $h(v_iv_i')$ (and possibly $h_i$) to $V(F)$, we add $v_i$ to $S_2.$ Furthermore, when for an edge $e\in E(T) \setminus E(P)$ the hyperedge $h(e)$ is added to $V(F)$, exactly one vertex (contained in $e$) is added to $S_2$. This gives the desired upper bound, where the term $n$ bounds the number of vertices added to $V(F)$ initially.
    
    \item 
    For \ref{i:density} first note that the number of vertices added to $V(F)$ at Step~(\ref{lab:bad tree}) of the algorithm is at least $v(F)-n$, as $n$ bounds the number of vertices added to $V(F)$ in the beginning of the algorithm. 
    Now, for each $v_i\in B'$, we either add just $h(v_i'v_i)$ to $V(F)$ and the two edges $\{h(v_i'v_i),h(v_i''v_i')\}$ to $F$, or we add both $h(v_i'v_i)$ and $h_i$ to $V(F)$, but then we add all three edges $\{h(v_i'v_i),h(v_i''v_i')\}$,$\{h(v_i'v_i),h_i\}$ and $\{h_i,h(e_i)\}$ to $F$. This means that for every two vertices added to $V(F)$, we add at least three edges to $F$, and in the observed case we added at least $|B'|$ vertices to $V(F)$. The total number of other vertices which are added to $V(F)$ in Step~\ref{lab:bad tree} is $|E(T) \setminus E(P)| \le |T_1\cup T_2|$ which (as we already explained in (D)) is at most $|B'|/100$.
    Hence, if we added $t$ vertices to $V(F)$ in a round, we added at least $\frac{3}{2}(t-t/100)\geq 1.45t$ edges to $F$, which completes the proof of the first part of \ref{i:density}.
    
    \item Now, it is left to show that there are no sunflower cycles in $F$. Look at $F$ at an arbitrary point of the algorithm, and assume that it has no sunflower cycle; we now show that after an additional round no sunflower cycle is created. Trivially, in the beginning $F$ consists only of isolated vertices, hence it has no sunflower cycles.
    
    Now we consider Step~\ref{lab:bad tree}, which is the only step where edges are added to $F$.
    All we need to show is that the edges in $\cI$ touching the vertex $h(v_iv_i')$ which are added to $F$ have different labels, as well as that the two edges touching $h_i$ have different labels if $h_i$ is also added to $V(F)$ in this step, hence no sunflower cycle can be closed by these new edges, which would complete the proof. 
    
    First we consider $h(v_iv_i')$. The label of $(h(v_iv_i'),h(v_i'v_i''))$ is $v_i'$, so if $(h(v_iv_i'),h_i)$ also had the label $v_i'$, then $h_i$ also contains $v_i'$ and hence it intersects $h(v_i'v_i'')$. Now, $h_i$ also intersects $h(e)$ for some edge $e_i$ on the path containing the root and $v_i'$, and by Observation~\ref{obs:ruin} we know that $e\neq v_i'v_i''$, so $h_i$ intersects that path in at least two vertices, a contradiction with the goodness of the tree $T$.
    
    Now, let us argue that if $h_i$ is added to $V(F)$, then the two edges added with it have different labels. Let the label of $(h(v_iv_i',h_i)$ be $x$. First, note that $x\neq v_i'$ by the same argument as in the previous passage. Now, if the label of the edge $(h_i,h(e_i))$ is also $x$, then the hyperedges $h(e_i)$ and $h(v_iv_i')$ intersect in $x$, which leads to a contradiction as the only hyperedge on the path from $v_i'$ to the root which $h(v_iv_i')$ intersects is $h(v_i'v_i'')$, and they only share vertex $v_i'$.
    
    \item For \ref{i:length}, note that each time we remove a vertex from $P$, at least one vertex is added to $S_1$.
    Then, if at some point $P$ has less than $\frac{2n}{L+1}$ vertices, then $|S_1|\geq \frac{2n}{L}-\frac{2n}{L+1}>\frac{n}{L^2}$,
    hence our algorithm would have stopped earlier. 
    \vspace{-0.9cm}
 \end{itemize}
 
 \end{proof}

\begin{lemma}\label{lem:DFS tree}
The algorithm above terminates when $|T|\geq \alpha N/10^9$, giving a good tree $T$ in $G'$ consisting of a path $P$ of length at least $\frac{2n}{L+1}$, and two trees $T_1$ and $T_2$ attached to its endpoints, and of depth at most $\log N$. Furthermore, $|T| \le \alpha N / 10^5$ and $|T_1|,|T_2|\geq \alpha N/10^{15}$.
\end{lemma}

\begin{proof}
Recall that  $\alpha=10^{-6}C^{-3}s^{-8}$ and $N = 10^{100}k^2C^6s^{14}n$, while $L_1=6$, $L_2=5$ and $L_3\leq 2^k+1$, and $L$ is equal to one of those three depending on which statement we prove (see definition of host graph in the end of Section~\ref{sec:host}). Also recall that $\delta(G')\geq 10^{15}L^2$, as mentioned in the end of Section~\ref{sec:setup}.
    Suppose that the algorithm terminates when $|S_1|\geq n/L^2 < \alpha N / 10^{20}$ (which implies that at that point $|S_1|\leq n/L^2+\alpha N/10^9$ because right before terminating, $|S_1|$ increased by at most $|T|\leq \alpha N/10^9$). We also have $|S_2| < 100n$. Since $S_1$ was increased right before terminating in Step~\ref{lab:no expansion},
    at the end of the observed round, $P,S_1,S_2$ and $U$ cover the vertices of $G'$.
    By \ref{i:bad expansion} there is a set $S\subset S_1$ with $|S|\geq |S_1|/10^5$ with $|N_U(S)|\leq 5\cdot10^8|S|\leq 5\cdot 10^{13}|S_1|$. Since $|P|+|S_2|< 101n \le 101L^2 |S_1|$, we get that $G'[P\cup S_1\cup S_2\cup N_U(S)]$ is a subgraph of $G'$ on at most $10^{14} L^2 |S_1| \leq \alpha N$ vertices, spanning at least $\delta(G') |S_1| / 2 > 2\cdot 10^{14} L^2 |S_1|$ edges, a contradiction with the local density property of $G\supset G'$ \ref{no-dense-aux-graph-set}.
    
    Next, suppose that the algorithm terminates when $|S_2|\geq  100n$. But then, by \ref{i:size of F} and \ref{i:density}, $F$ has at least $25n$ vertices, and at least $1.45(v(F)-n)> \frac{4}{3}v(F)$ edges, which is again a contradiction with the local density condition of $\cI\supset F$ stated as \ref{no-dense-set}, because $v(F)\leq 100n+\alpha N/2<\alpha N$ (in the last step before terminating the number of vertices which are added to $S_2$ is at most $10^8 |T_1\cup T_2|<\alpha N/2$).
    
    Hence, the algorithm terminates when $|T|\geq \alpha N/10^9$. By construction, $T_1$ and $T_2$ are of depth at most $\log_{10^4} (\alpha N) < \log N$.
    Additionally, we have that the bound on the ratio of the sizes of the trees $T_1$ and $T_2$ is $\max\{\frac{|T_1|}{|T_2|},\frac{|T_2|}{|T_1|}\}\leq 2 \cdot 10^4$, which gives the lower bound on the sizes of the trees, since $T=T_1+P+T_2$ and $|P|\leq n$. Since $|T|$ increases in each round by at most a factor of $10^4,$ $|T| \le \alpha N / 10^5$ holds. Finally, by \ref{i:length}, $P$ is of length at least $\frac{2n}{L+1}$, which finishes the proof.
\end{proof}
Let $R_1$ be the set of vertices in $X$ which are added to $T_1$ in the last $2\log_{10^4} (C^6s^6)$ rounds of the algorithm. Recall that each $T_i$ grows by factor roughly $10^4$ every other round. It follows that $|V(T_1)-R_1|\leq \frac{2|T_1|}{C^6s^6}\leq 10^{90}k^2n/C^3\leq n/2$, where we used that $\alpha=10^{-6}C^{-3}s^{-8}$, $N = 10^{100}k^2C^6s^{14}n$, $C\geq 10^{20}k$ and $k$ is large enough.
We define $R_2$ analogously for the tree $T_2$, and again we have $|V(T_2)-R_2|\leq n/2$.

\subsection{Enlarging the trees and finishing the proof}\label{sec:enlarging}
In this section we complete the proofs of Theorems~\ref{thm:main} and \ref{thm:non-induced}. Recall that the tree $T$ from \Cref{lem:DFS tree} lives inside $G'$, and $G'$ is a subgraph of a $\gamma$-expanding graph $G_{red}'$, with $\gamma=\frac{1}{k^2Cs}$. Recall again that  $\alpha=10^{-6}C^{-3}s^{-8}$ and $N = 10^{100}k^2C^6s^{14}n$.

For a subset $S\subseteq V(\mathcal I)$, denote by $N^\cI(S)$ all the vertices in the graph $G$ which are contained in a hyperedge $h$ which has distance at most 2 to $S$ in $\mathcal{I}$.  Note that for every $S\subseteq \mathcal I$, in the graph $G$ we always have $|N^\cI(S)|\leq 2(8Cs)^2s^3 |S|$, since the maximum degree of $H$ is at most $8Cs$ and the size of each hyperedge is $s$. Let $S$ be the set of hyperedges $h(e)$ where $e\in T-R_1-R_2$. Then $|S|= |V(T_1)-R_1|+|V(T_2)-R_2|+|P| \leq n/2+n/2+n=2n$.

We now use the expanding properties of $\Gr'$ to find a short path between the sets $R_1$ and $R_2$ and thus we close a cycle together with a subpath of $T.$ Crucially, we make sure that this path avoids $N^\calI(S)$ which, together with the fact that $H$ has large girth, guarantees that the cycle we find is good. The formal proof follows.

Let $X_t \coloneqq R_t$ for $t \in \{1, 2\}.$ Since $G_{red}'$ is a $\gamma$-expander we can do the following.
 For each $t \in \{1, 2\},$ repeat the following as long as $|X_t| \le |V(G_{red}')|/2$. Let
\[
X_t:=X_t\cup N_{G_{red}'}(X_t)\setminus N^{\cI}(S).
\]
Since at each step $X_t$ increases by at least $\gamma|X_t|-2(8Cs)^2s^3|S|\geq \gamma \alpha N/10^{15}-256C^2s^5n\geq \gamma \alpha N/10^{16}$, after at most $\frac{10^{16}}{\gamma\alpha}$ steps, both $X_1$ and $X_2$ contain more than half of the vertices of $G_{red}'$. Thus, they intersect at some point of this procedure. Let $x$ be one of the vertices in $X_1\cap X_2$, at the time of the procedure when we first have $X_1\cap X_2\neq \emptyset$. Hence, we obtain a cycle $Q$ in $G_{red}'$ which passes through $T_1$ and $T_2$ (and its root), and contains vertex $x$, and has length between $\frac{2n}{L+1}\leq|P|$ and $\frac{2n}{L}+2\log N+2\frac{10^{16}}{\gamma\alpha}\leq \frac{2n}{L-1}$, where in the upper bound the first term is a bound on the length of $P$, the second bounds the sum of the depths of $T_1$ and $T_2$, and the third is the number of steps needed to reach $x$ from $R_1$ and $R_2$. This in turn will give rise to a red (induced) cycle of length $n$ in the host graph $\Gamma$. To prove this, we use Lemma~\ref{lem:cycle in host graph}, so we first show the existence of a collection of good paths contained in $Q$, so that every pair of edges in $Q$ is contained in one such path.

Let $P_1$ be the path $Q\cap (X_1\cup T\setminus R_2)$, and let us argue that this path is a good path. Since $T$ is a good tree, $Q\cap T$ is a good path. Furthermore, since the path $Q\cap X_1$ is of length at most $\frac{10^{16}}{\gamma\alpha}+\log_{10^4}(C^6s^6)<g$ where $g=(Cs)^{20}$ is the girth of $H$, by \Cref{lem:short is good} it is also a good path. To finish, we observe that any pair of edges $e,e'$ where $e$ has a vertex in $X_1-R_1$, and $e'$ has a vertex in $T-R_2$ is such that $e$ and $e'$ are at distance more than $2$ in $\mathcal I$. Indeed, by construction vertices in $X_1-R_1$ are outside of $N^{\mathcal I}(S)$ and so $h(e)$ and $h(e')$ do not intersect and there is no hyperedge in $H$ intersecting both $h(e)$ and $h(e').$

The path $P_2$ defined as $Q\cap (X_2\cup T\setminus R_1)$ is analogously good. Furthermore the path $P_3= Q\cap(X_1\cup X_2)$ is good by Lemma~\ref{lem:short is good}, as we only performed at most $\frac{10^{16}}{\gamma\alpha}$ steps to reach the final $X_1$ and $X_2$, so $Q\cap (X_1\cup X_2)$ is of length at most $2\frac{10^{16}}{\gamma\alpha} + 2 \log_{10^4}(C^6s^6) < g$.

Note that $P_1,P_2$ and $P_3$ cover every pair of edges in $Q$, so $Q$ is a good cycle. This completes the proof of \Cref{thm:main} and \Cref{thm:non-induced} by Lemma~\ref{lem:cycle in host graph}, recalling that $Q$ is of length between $\frac{2n}{L+1}$ and $\frac{2n}{L-1}$.

\section{Concluding remarks}
In this paper we have shown that the induced size-Ramsey number of even cycles satisfies $\hat r^k_{ind}(C_n)\leq O(k^{102})n$, while for odd cycles we have $\hat r^k_{ind}(C_n)\leq e^{O(k\log k)}n$. Both of the results improve upon the previous best known bound \cite{haxell1995induced} which had a tower type dependence on $k$. 
Using the same method, we prove essentially tight bounds for the size-Ramsey number of odd cycles, namely $\hat r^k(C_n)\leq e^{O(k)}$. Being more careful with the argument, we could derive a more precise bound also in the even case for (non-induced) size-Ramsey numbers, though we certainly cannot come very close to the best known lower bound, which is of order $k^2n$ \cite{krivelevich2019expanders,dudek2017some}. 

As an interesting open problem, we conjecture that an exponential bound in $k$ is also sufficient in the induced case.

\begin{conjecture}
$\hat r^k_{ind}(C_n)\leq e^{O(k)}n$ for odd $n$.
\end{conjecture}
In order to show this, it would be sufficient to prove that there is a graph with $e^{O(k)}$ edges which contains an induced monochromatic odd cycle of length at least $5$ in any $k$-coloring of its edges. Indeed, then we could use this graph as the gadget graph, and get the desired bound with the same proof.
We conjecture that such a gadget graph exists.

\begin{conjecture} \label{conj:induced-odd-cycle}
For every integer $k$, there is a graph $G$ with $e^{O(k)}$ edges which, for any $k$-coloring of its edges, contains a monochromatic odd cycle of length at least 5 as an induced subgraph.
\end{conjecture}

Determining whether the multicolor Ramsey number of a fixed odd cycle, say $C_3$, is exponential in the number of colors is a notoriously hard open problem. However, Conjecture~\ref{conj:induced-odd-cycle} could be substantially easier since it requires us to find a monochromatic cycle of any odd length and can be viewed as an induced analogue of the simple Lemma~\ref{lem:any-odd-cycle}.

\bibliographystyle{plain}

\begin{thebibliography}{10}

\small

\bibitem{alon1994explicit}
Noga Alon.
\newblock Explicit {R}amsey graphs and orthonormal labelings.
\newblock {\em The Electronic Journal of Combinatorics}, pages R12--R12, 1994.

\bibitem{beck1983size}
J{\'o}zsef Beck.
\newblock On size {R}amsey number of paths, trees, and circuits. I.
\newblock {\em Journal of Graph Theory}, 7(1):115--129, 1983.

\bibitem{Berger21}
S\"{o}ren Berger, Yoshiharu Kohayakawa, Giulia~Satiko Maesaka, Ta\'{i}sa Martins,
  Walner Mendon\c{c}a, Guilherme~Oliveira Mota, and Olaf Parczyk.
\newblock The size-{R}amsey number of powers of bounded degree trees.
\newblock {\em Journal of the London Mathematical Society}, 103(4):1314--1332,
  2021.

\bibitem{bondy1974cycles}
John~A. Bondy and Mikl{\'o}s Simonovits.
\newblock Cycles of even length in graphs.
\newblock {\em Journal of Combinatorial Theory, Series B}, 16(2):97--105, 1974.

\bibitem{CFS}
David Conlon, Jacob Fox and Benny Sudakov, Recent developments in graph Ramsey theory. In: {\it Surveys in Combinatorics} 2015, Cambridge University Press, 2015, 49--118.

\bibitem{Deuber}
Walter Deuber, A generalization of Ramsey's theorem. In: {\it Infinite
and Finite Sets}, Vol. 1, Colloquia Mathematica Societatis J\'anos
Bolyai, Vol. 10, North-Holland, Amsterdam/London, 1975, 323--332.

\bibitem{draganic2022short}
Nemanja Dragani{\'c}, Stefan Glock, and Michael Krivelevich.
\newblock Short proofs for long induced paths.
\newblock {\em Combinatorics, Probability and Computing}, 31(5):870--878, 2022.

\bibitem{draganic2022rolling}
Nemanja Dragani{\'c}, Michael Krivelevich, and Rajko Nenadov.
\newblock Rolling backwards can move you forward: on embedding problems in
  sparse expanders.
\newblock {\em Transactions of the American Mathematical Society}, 2022.

\bibitem{draganic2022size}
Nemanja Dragani{\'c} and Kalina Petrova.
\newblock Size-{R}amsey numbers of graphs with maximum degree three.
\newblock {\em arXiv preprint arXiv:2207.05048}, 2022.

\bibitem{dudek2015alternative}
Andrzej Dudek and Pawe{\l} Pra{\l}at.
\newblock An alternative proof of the linearity of the size-{R}amsey number of
  paths.
\newblock {\em Combinatorics, Probability and Computing}, 24(3):551--555, 2015.

\bibitem{dudek2017some}
Andrzej Dudek and Pawe{\l} Pra{\l}at.
\newblock On some multicolor {R}amsey properties of random graphs.
\newblock {\em SIAM Journal on Discrete Mathematics}, 31(3):2079--2092, 2017.

\bibitem{erdos1975problems}
Paul Erd\H{o}s.
\newblock Problems and results on finite and infinite graphs.
\newblock In {\em Recent Advances in Graph Theory. Proceedings of the Second
  Czechoslovak Symposium}, pages 183--192. Academia, Prague, 1975.
  
\bibitem{erdHos1978size}
Paul Erd\H{o}s, Ralph~J. Faudree, Cecil~C. Rousseau, and Richard~H. Schelp.
\newblock The size {R}amsey number.
\newblock {\em Periodica Mathematica Hungarica}, 9(1-2):145--161, 1978.

\bibitem{ErHaPo}
Paul Erd\H{o}s, Andr\'{a}s Hajnal, and Lajos P\'osa. Strong embeddings of graphs into colored graphs. In: {\it Infinite and Finite Sets}, Vol. 1,
Colloquia Mathematica Societatis J\'anos Bolyai, Vol. 10,
North-Holland, Amsterdam/London, 1975, 585--595.

\bibitem{erdos-renyi}
Paul Erd\H{o}s and Alfr\'{e}d R\'{e}nyi.
\newblock On a problem in the theory of graphs.
\newblock {\em A Magyar Tudományos Akadémia. Matematikai Kutató
  Intézetének Közleményei}, 7:623--641, 1962.

\bibitem{fox2008induced}
Jacob Fox and Benny Sudakov.
\newblock Induced {R}amsey-type theorems.
\newblock {\em Advances in Mathematics}, 219(6):1771--1800, 2008.

\bibitem{friedman1987expanding}
Joel Friedman and Nicholas Pippenger.
\newblock Expanding graphs contain all small trees.
\newblock {\em Combinatorica}, 7(1):71--76, 1987.

\bibitem{haxell1995induced}
Penny~E. Haxell, Yoshiharu Kohayakawa, and Tomasz {\L}uczak.
\newblock The induced size-{R}amsey number of cycles.
\newblock {\em Combinatorics, Probability and Computing}, 4(3):217--239, 1995.

\bibitem{JLR}
Svante Janson, Tomasz {\L}uczak and Andrzej  Ruci\'nski,
{\it Random graphs}, John Wiley \& Sons, 2011.

\bibitem{javadi2019size}
Ramin Javadi, Farideh Khoeini, Gholam~Reza Omidi, and Alexey Pokrovskiy.
\newblock On the size-{R}amsey number of cycles.
\newblock {\em Combinatorics, Probability and Computing}, 28(6):871--880, 2019.

\bibitem{javadi2023multicolor}
Ramin Javadi and Meysam Miralaei.
\newblock The multicolor size-{R}amsey numbers of cycles.
\newblock {\em Journal of Combinatorial Theory, Series B}, 158:264--285, 2023.

\bibitem{kamcev2021size}
Nina Kam\v{c}ev, Anita Liebenau, David~R. Wood, and Liana Yepremyan.
\newblock The size {R}amsey number of graphs with bounded treewidth.
\newblock {\em SIAM Journal on Discrete Mathematics}, 35(1):281--293, 2021.

\bibitem{kohayakawa1998induced}
Yoshiharu Kohayakawa, Hans~J\"{u}rgen Promel, and Vojt\v{e}ch R\"{o}dl.
\newblock Induced {R}amsey numbers.
\newblock {\em Combinatorica}, 18(3):373--404, 1998.

\bibitem{kohayakawa2011sparse}
Yoshiharu Kohayakawa, Vojt{\v{e}}ch R{\"o}dl, Mathias Schacht, and Endre
  Szemer{\'e}di.
\newblock Sparse partition universal graphs for graphs of bounded degree.
\newblock {\em Advances in Mathematics}, 226(6):5041--5065, 2011.

\bibitem{krivelevich2019expanders}
Michael Krivelevich.
\newblock Expanders---how to find them, and what to find in them.
\newblock In: {\em Surveys in combinatorics 2019}, volume 456 of {\em London
  Mathematical Society Lecture Note Series}, pages 115--142. Cambridge
  University Press, Cambridge, 2019.

\bibitem{krivelevich2019long}
Michael Krivelevich.
\newblock Long cycles in locally expanding graphs, with applications.
\newblock {\em Combinatorica}, 39(1):135--151, 2019.

\bibitem{krivelevich2006pseudo}
Michael Krivelevich and Benny Sudakov.
\newblock Pseudo-random graphs.
\newblock In: {\em More sets, graphs and numbers}, pages 199--262. Springer,
  2006.

\bibitem{Ramsey1930}
Frank~P. Ramsey.
\newblock On a problem of formal logic.
\newblock {\em Proceedings of the London Mathematical Society}, 30(4):264--286,
  1929.

\bibitem{Rodl73}
Vojt\v{e}ch R\"{o}dl.
\newblock The dimension of a graph and generalized Ramsey theorems, \newblock Master's thesis, Charles University, 1973.

\bibitem{rodl2000size}
Vojt\v{e}ch R\"{o}dl and Endre Szemer{\'e}di.
\newblock On size {R}amsey numbers of graphs with bounded degree.
\newblock {\em Combinatorica}, 20(2):257--262, 2000.

\bibitem{tikhomirov2022bounded}
Konstantin Tikhomirov.
\newblock On bounded degree graphs with large size-{R}amsey numbers.
\newblock {\em arXiv preprint arXiv:2210.05818}, 2022.

\end{thebibliography}

\section{Appendix}
\thmmichaelrestate*
\begin{proof}
    Set $\delta = \frac{c_1 - c_2}{2 \lceil \log_2 \frac{1}{\beta} \rceil}$ and $d_i = c_1 - i\delta,$ for $i \ge 0.$ The proof proceeds in several iterations. We start with $i = 0,$ $G_0 = G.$ Suppose we are at iteration $i \ge 0$ and that the graph $G_i$ satisfies $|E(G_i)| / |V(G_i)| \ge d_i,$ which holds for $i = 0$ by assumption. Let $H_i = (U_i, F_i)$ be an inclusion-wise minimal induced subgraph of $G_i$ for which $|F_i| / |U_i| \ge d_i$ (which exists since $G_i$ satisfies this condition). Note that every subset $W \subseteq U_i$ touches at least $d_i |W|$ edges as otherwise we could remove $W$ from $U_i$ to obtain a subgraph with average degree at least $2d_i$  contradicting the minimality of $H_i.$ If
    \begin{equation} \label{eq:W}
        \text{there exists } W \subseteq U_i, \beta n \le |W| \le |U_i| / 2, \text{ s.t. } W \text{ spans at least } d_{i+1}|W| \text{ edges in } H_i,
    \end{equation}
    then we update $G_{i+1} = G[W], i \coloneqq i+1,$ otherwise we halt the process.
    Observe that if we proceed to the next iteration as described above, then $|V(G_{i+1})| \le |V(G_i)| / 2$ and $|E(G_{i+1})| / |V(G_{i+1})| \ge d_{i+1}.$ Hence, if we reach iteration $i = \lceil \log_2(1/\beta) \rceil,$ we arrive at a subgraph $G_i$ of size at most $\beta n$ satisfying $|E(G_i)| / |V(G_i)| \ge d_i = \frac{c_1 + c_2}{2},$ contradicting the second assumption. Therefore, the process stops at some $i \le \lceil \log_2(1/\beta) \rceil - 1,$ implying that \eqref{eq:W} is not satisfied for $H_i = (U_i, F_i).$ Since $|E(H_i)| / |V(H_i)| \ge d_i \ge (c_1 + c_2) / 2,$ we have that $|V(H_i)| \ge \beta n$ by the second assumption. 
    
    Finally, we verify that $H_i$ is a $\gamma$-expander so we can take $G^* = H_i.$ Indeed, let $W \subseteq U_i, \, \beta n \le |W| \le |U_i| / 2.$ Recall that $H_i$ does not satisfy \eqref{eq:W}, so $W$ spans at most $d_{i+1}|W|$ edges, yet by our choice of $H_i,$ $W$ touches at least $d_i |W|$ edges in $H_i.$ Thus, $(W, U_i \setminus W) \ge (d_i - d_{i+1}) |W| = \delta |W|.$ Now, suppose $|W| \le \beta n.$ By our second assumption, $|W|$ spans at most $c_2 |W|$ edges in $H_i,$ but touches at least $d_i |W| \ge \left(\frac{c_1 + c_2}{2} + \delta\right) \ge (c_2 + \delta) |W|$ edges in $H_i.$ Again, we conclude that $e(W, U_i \setminus W) \ge \delta |W|.$ Recalling the maximum degree condition, we see that $|N_{H_i}(W)| \ge e(W, U_i \setminus W) / \Delta \ge \frac{\delta}{\Delta}|W| = \gamma |W|,$ implying that $H_i$ is a $\gamma$-expander, as required.
\end{proof}

\end{document}